\documentclass[11pt]{article}
\usepackage{amssymb, amsmath, amsthm,srcltx}
\usepackage{hyperref}
\usepackage{graphicx}
\usepackage{color}
\newtheorem{thm}{Theorem}[section]
\newtheorem{lem}{Lemma}[section]
\newtheorem{cor}[lem]{Corollary}
\newtheorem{prop}[thm]{Proposition}
\theoremstyle{definition}
\newtheorem{defn}[thm]{Definition}
\theoremstyle{remark}
\newtheorem{rem}[thm]{Remark}
\numberwithin{equation}{section}

\newcommand{\E}{\mathbf{E}\,}

\newcommand{\Tr}{\mathrm{Tr}\;\!}

\newcommand{\re}{\mathrm{Re}\;\!}
\newcommand{\im}{\mathrm{Im}\;\!}

\newenvironment{Proof of}{\removelastskip\par\medskip
\noindent{\em Proof of} \rm}{\penalty-20\null\hfill$\square$\par\medbreak}

\begin{document}
\date{April 15, 2011}
\title{\bf On the  Asymptotic Spectrum of Products of Independent Random Matrices.}

\author{{\bf F. G\"otze}\\{\small Faculty of Mathematics}
\\{\small University of Bielefeld}\\{\small Germany}
\and {\bf A. Tikhomirov}$^{1}$\\{\small Department of Mathematics
}\\{\small Komi Research Center of Ural Branch of RAS,}\\{\small Syktyvkar State University}
\\{\small  Russia}}

\maketitle
 \footnote{$^1$Partially supported by RF grant of the leading scientific schools
NSh-638.2008.1. Partially supported by  RFBR, grant  N 09-01-12180 and RFBR--DFG, grant N 09-01-91331.
Partially supported by CRC 701 ``Spectral Structures and Topological
Methods in Mathematics'', Bielefeld}



\begin{abstract}We consider products of independent random  matrices 
with independent entries. 
The limit distribution of the expected empirical distribution of eigenvalues of such products is computed. 
Let $X^{(\nu)}_{jk},{  }1\le j,r\le n$, $\nu=1,\ldots,m$ be mutually independent complex random variables
with $\E X^{(\nu)}_{jk}=0$ and $\E {|X^{(\nu)}_{jk}|}^2=1$. Let 
$\mathbf X^{(\nu)}$ denote  an $n\times n$  matrix with entries $[\mathbf X^{(\nu)}]_{jk}=\frac1{\sqrt{n}}X^{(\nu)}_{jk}$, 
for $1\le j,k\le n$.
  Denote by $\lambda_1,\ldots,\lambda_n$ the eigenvalues  of
the random  matrix
$\mathbf W:= \prod_{\nu=1}^m\mathbf X^{(\nu)}$
and define its  empirical spectral distribution  by
$$
\mathcal F_n(x,y)=\frac1n\sum_{k=1}^n\mathbb I\{\re{\lambda_k}\le x,\im{\lambda_k\le y}\},
$$
where $\mathbb I\{B\}$ denotes the indicator of an event $B$.
 We prove that
 the expected spectral distribution $F_n^{(m)}(x,y)=\E \mathcal F_n^{(m)}(x,y)$ converges
 to the  distribution function  $G(x,y)$ corresponding to  the $m$-th power of
 the  uniform distribution on the unit disc in  
the plane $\mathbb R^2$.

\end{abstract}

\maketitle
\markboth{  F. G\"otze, A.Tikhomirov}{Product of random matrices}

\section{Introduction}
Let $m\ge1$ be a fixed integer. For any $n\ge 1$ consider mutually independent  identically distributed (i.i.d.)
 complex random variables $X^{(\nu)}_{jk},\quad{  }1\le j,k\le n$, $\nu=1,\ldots,m$
with $\E X^{(\nu)}_{jk}=0$ and $\E {|X^{(\nu)}_{jk}|}^2=1$ defined on a common probability space $(\Omega_n, 
\mathbb F_n,\Pr)$. Let
$\mathbf X^{(\nu)}$ denote  an $n\times n$  matrix with entries $[\mathbf X^{(\nu)}]_{jk}=\frac1{\sqrt{n}}
X^{(\nu)}_{jk}$, for $1\le j,k\le n$.
Denote by $\lambda_1,\ldots,\lambda_n$ the eigenvalues  of
the random  matrix
$\mathbf W:= \prod_{\nu=1}^m\mathbf X^{(\nu)}$
and define its empirical spectral distribution  function by
$$
\mathcal F_n(x,y)=\frac1n\sum_{k=1}^n\mathbb I\{\re{\lambda_k}\le x,\im{\lambda_k}\le y\},
$$
where $\mathbb I{\{B\}}$ denotes the indicator of an event $B$.
We shall investigate the
convergence of the expected spectral distribution $F_n(x,y)=\E \mathcal F_n(x,y)$
 to
the  distribution function  $G(x,y)$ corresponding to the $m$-th power of uniform distribution on the unit disc in the plane $\mathbb R^2$ with Lebesgue-density
$$
g(x,y)=\frac{1}{\pi m (x^2+y^2)^{\frac{m-1}{m}}}I\{x^2+y^2\le1\}.
$$
We consider  the  Kolmogorov distance between the distributions
$ F_n(x,y)$ and $G(x,y)$
$$
\Delta_n:=\sup_{x,y}|F_n(x,y)-G(x,y)|.
$$
The main result of this paper is the following
\begin{thm}\label{main}Let $\E X^{(\nu)}_{jk}=0$, $\E|X^{(\nu)}_{jk}|^2=1$.
Then, for any fixed $m\ge 1$,
\begin{equation}\notag
\lim_{n \to \infty} \sup_{x,y}|F_n(x,y)-G(x,y)|= 0.
\end{equation}
\end{thm}
The result holds in the non-i.i.d. case too.
\begin{thm}\label{Lindeberg}
 Let $\E X^{(\nu)}_{jk}=0$, $\E|X^{(\nu)}_{jk}|^2=1$ and assume that
the random variables $X_{jk}^{(\nu)}$ have  uniformly  integrable second moments, i. e.
\begin{equation}
 \max_{\nu, j,k}\E|X_{jk}^{(\nu)}|^2I\{|X_{jk}^{(\nu)}|>M\}\to0\quad\text{as}\quad M\to\infty.
\end{equation}

Then
for any fixed $m\ge 1$,
\begin{equation}\notag
\lim_{n \to \infty} \sup_{x,y}|F_n(x,y)-G(x,y)|= 0.
\end{equation}

\end{thm}

\begin{defn}
 Let $\mu_n(\cdot)$ denote the  empirical spectral measure of an $n\times n$ random matrix $\bold X$ and
let $\mu(\cdot)$ denote the  uniform distribution on the unit disc in the complex plane $\mathbb C$.
We say that  the circular law holds for  random matrices $\bold X$  if
$\E\mu_n(\cdot)$ converges weakly  to the measure $\mu(\cdot)$ in the complex plane $\mathbb C$.
\end{defn}

\begin{rem}For $m=1$ we recover the  well-known circular law for random  matrices
 \cite{GT:2010}, \cite{TV:2010}.
\end{rem}

Theorems \ref{main} and \ref{Lindeberg} describe the asymptotics  of
the  spectral distribution of a product of $m$ independent random matrices. This generalizes the result of \cite{GT:2010} and \cite{TV:2010}. 

\subsection {Discussion of results} The proof of these results are based 
on the author's investigations on  asymptotics of the singular spectrum of product and powers of random matrices with independent entries 
(see \cite{AGT:2010}, \cite{AGT:2010a}, \cite{AGT:2010b}).
Our results give a full  description of the complex  spectral  distribution of products of large random matrices. 
The results  mentioned on the   asymptotic distribution of the singular spectrum of products of independent random matrices where already obtained some time
ago by the authors, see {\cite{AGT:2010a}, \cite{AGT:2010}}.  
Related previous  results concerned bounds for the expectation of the operator
norm of two independent matrices, see Bai (1986), \cite {Bai:1986}.   
In Bai (2007), \cite{Bai:2007},  the asymptotic distribution of the product of a
sample covariance matrix and an independent  Wigner matrix is investigated.
Some questions about the asymptotic distribution of products and powers of
random matrices were  studied in  Free Probability. 
For example, in Banica et al. (2008), \cite{Capitaine:2008},
the asymptotic distribution of the singular value distribution 
 of the  product of squares  of independent Gaussian  random matrices is
 determined.  In Speicher (2008),  \cite{Speicher:2008}, the same asymptotic
 distribution has been obtained for the singular value distribution 
of   products and powers of random matrices. 

A related result for norms has been obtained by Haagerup and Torbj{\o}nson
\cite{HT:2005},  who proved that if $\mathbf X^{(1)}, \ldots,\mathbf X^{(m)}$ 
is  a system of independent Gaussian random matrices and $x_1,\ldots, x_r $
is a corresponding  semi-circular system in a $C^*$ probability space, then
for every polynomial $p$ in $r$ non commuting variables we have an asymptotic 
norm equality
\begin{equation}
 \lim_{n\to\infty}\|p(\mathbf X^{(1)},\ldots,\mathbf X^{(m)})\|=\|p(x_1,\ldots,x_m)\|
\end{equation}
which holds almost surely.

Our result on  the asymptotic
distribution of the  complex eigenvalues of products of large (non-Hermitian
and non Gaussian) random matrices seemed to be new. After finishing this paper
we learned though that  the case of products of {\it Gaussian} had been 
studied  by Burda et al. (2010),  \cite{Burda-Janik-Waclaw:2010},
with our  main result stated as  conjecture,  supported by simulations.

 We expect that
results of this  type will be useful for the analysis 
of some models of wireless communication. See for instance, \cite{Hwang:2009}.

The results of both Theorems \ref{main} and \ref{Lindeberg} may be considered
 as generalizations of the circular law,
see e.g. \cite{GT:2010} for some history on  the circular law and its proof.

To prove the claim of  both Theorems \ref{main} and \ref{Lindeberg} we use the logarithmic potential approach as  in \cite{GT:2010}.
We may divide this approach into two parts. The first part deals with 
the  investigation of the asymptotic distribution of the singular values 
of shifted matrices $\mathbf W(z):=\mathbf W-z\mathbf I$. 
To study these distributions  we use the  method developed
 in \cite {AGT:2010b} for the case $z=0$.
The other part  will be the investigation of small singular values of matrices $\mathbf W(z) $ for any 
$z\in\mathbb C$. This problem  may be  divided again in two parts. 
The first part consists of  the investigation of smallest singular values. 
Here we may use  our results in  \cite{GT:2010}  or the results in \cite{TV:2010}.
The second part  deals with  the investigation of the  singular values between
the smallest one to the 
$j$th smallest one, where  $j\ge n-n^{\gamma}$ for some $0<\gamma<1$. 
Here we use a  modification of techniques  of Tao and Vu in \cite{TV:2010}.

In the remaining parts of the  paper we give  the proof of Theorem \ref{Lindeberg}.  Theorem \ref{main}
follows immediately from \ref{Lindeberg}. We shall use the logarithmic
potential method which is outlined in detail in \cite{GT:2010}.

In Section \ref{convergence} we derive  the 
 approximation of the  singular measure of the shifted matrix 
$\mathbf W(z)$ for any $z\in\mathbb C$. This allows us to prove the
convergence of the empirical spectral
 measure of the matrix $\bold W(z)$ to the
 corresponding limit  measure in $\mathbb R^2$.
The convergence is proved in Section \ref{proof}.

{In the what follows we shall denote by $C$ and $c$ or $\delta,\rho, \eta$ (without indices) 
some general absolute constant which may be change from one line to next one.
To specify a constant we shall use subindices. By $I\{A\}$ we shall denote the indicator of an event $A$.
For any matrix $\bold G$ we denote the Frobenius norm by $\|\bold G\|_2$ and
we denote by $\|\bold G\|$ its operator norm.}

%

{\bf Acknowledgment}. The authors would like to thank  Sergey  Bobkov
for helpful discussions concerning Maurey's result and Gernot  Akemann for
drawing our attention to the paper \cite{Burda-Janik-Waclaw:2010}.

\section{Auxiliary Results}In this Section we describe a
 symmetrization of one-sided distribution and  a special
representation of symmetrized distributions of squares  singular
values of random matrices and prove some lemmas about a truncation
of entries of random matrices.
\subsection{Symmetrization}We shall use the following  ``symmetrization'' of one-sided distributions. Let $\xi^2$
 be a positive random variable with distribution function $F(x)$. Define
$\widetilde \xi:=\varepsilon\xi$ where $\varepsilon$ is a
 Rademacher random variable with
$\Pr\{\varepsilon=\pm1\}=1/2$ 
which is independent of $\xi$.  Let $\widetilde F(x)$ denote
the distribution function of $\widetilde \xi$. It satisfies the
equation
\begin{equation}\label{sym}
\widetilde F(x)=1/2(1+\text{\rm sign} \{x\}\,F(x^2)),
\end{equation}
We apply this symmetrization to the distribution of the squared
singular values of the matrix $\mathbf W(z)$. Introduce the following matrices
\begin{align}\notag
\mathbf V:=\left(\begin{matrix}&\mathbf W &\mathbf O\\
&\mathbf O &\mathbf W^*\end{matrix}\right),\quad\mathbf J(z):=\left(\begin{matrix}&\mathbf O &z\mathbf I\\
&\overline z\mathbf I\end{matrix}\right),\quad \mathbf J:=\mathbf J(1).
\end{align}
Here and in the what follows  $\mathbf A^*$ denotes the adjoined (transposed and complex conjugate)  matrix $\mathbf A$ 
and $\mathbf O$ denotes the matrix with zero-entries.
Consider matrix
\begin{equation}
\mathbf V(z):=\mathbf V\mathbf J-\mathbf J(z).
\end{equation}
Note that ${\mathbf V(z)}$ is a Hermitian matrix. 
The eigenvalues of the matrix ${\mathbf V(z)}$ are $-s_1,\ldots,-s_n,$ $s_n,\ldots,s_1$.
Note that the symmetrization of the distribution function $\mathcal
F_n(x,z)$ is a function $\widetilde{\mathcal F}_n(x,z)$ is the empirical 
distribution function of the  non-zero eigenvalues of the  matrix ${\mathbf V(z)}$. By (\ref{sym}), we have
\begin{equation}\notag
\Delta_n=\sup_x|\widetilde F_n(x,z)-\widetilde
G(x,z)|,
\end{equation}
where $\widetilde F_n(x,z)=\E\widetilde{\mathcal F}_n(x,z)$ and
$\widetilde G(x,z)$ denotes the  symmetrization of the distribution function
$G(x,z)$.

\subsection{Truncation}
We shall now modify the random matrix $\bold X^{(\nu)}$ by truncation of its entries. In this section we shall
 assume that the random variables $X_{jk}^{(\nu)}$ satisfy the following Lindeberg condition: for any $\tau>0$
\begin{equation}\label{lindeb}
 L_n(\tau)=\max_{1\le\nu\le m}\frac1{n^2}\sum_{j,k=1}^n\E|X_{jk}^{(\nu)}|^2
I\{|X_{jk}^{(\nu)}|\ge\tau\sqrt n\}\to0,\quad\text{as}\quad n\to\infty.
\end{equation}
It is straightforward to check that this Lindeberg condition follows from uniform integrability.
  We introduce the 
random variables
$X^{(\nu,c)}_{jk}=X^{(\nu)}_{jk}I_{\{|X^{(\nu)}_{jk}|\le
c\tau_n\sqrt n\}}$ with $\tau_n\to0$  and the  matrices $\mathbf X^{(\nu,c)}=\frac1{\sqrt
{n}}({X^{(\nu,c)}_{jk}})$ and $ \mathbf W^{(c)}:= \prod_{\nu=1}^m{{\mathbf
X^{(\nu,c)}}}$. Denote by
$s_1^{(c)}\ge\ldots\ge s_n^{(c)}$ the singular values  of the
random  matrix $\mathbf W^{(c)}-z\mathbf I $. Let $\mathbf V^{(c)}:=\left(\begin{matrix}&\mathbf
W^{(c)} &\mathbf O\\&\mathbf O &{\mathbf W^{(c)}}^*\end{matrix}\right)$. We
define  the empirical distribution of the matrix $\mathbf V^{(c)}(z)=\mathbf V^{(c)}\mathbf J-\mathbf J(z)$ by 
$\widetilde{\mathcal
F}_n^{(c)}(x)=\frac1{2n}\sum_{k=1}^nI{\{{s_k^{(c)}}\le
x\}}+\frac1{2n}\sum_{k=1}^nI{\{{-s_k^{(c)}}\le x\}}$. Let
$s_n(\alpha,z)$ and $s_n^{(c)}(\alpha,z)$ denote the  \nobreak{Stieltjes} transforms of
the distribution functions $\widetilde F_n(x)$ and $\widetilde
F_n^{(c)}(x)=\E\widetilde{\mathcal F}_n^{(c)}(x)$
respectively.
Define the resolvent matrices $\mathbf R=({\mathbf V}(z)-\alpha\mathbf I)^{-1}$ and
$\mathbf R^{(c)}=({\mathbf V}^{(c)}(z)-\alpha\mathbf I)^{-1}$, where $\mathbf I$ denotes the identity matrix of 
corresponding dimension. Note that
\begin{equation}\notag
 s_n(\alpha,z)=\frac1{2n}\E\Tr \mathbf R,\qquad\text{and}\qquad
s_n^{(c)}(\alpha,z)=\frac1{2n}\E\Tr \mathbf R^{(c)}.
\end{equation}
Applying the  resolvent equality
\begin{equation}\label{resolvent_equality}
 (\mathbf A+\mathbf B-\alpha\mathbf I)^{-1}=(\mathbf A-\alpha\mathbf I)^{-1}-(\mathbf A-\alpha\mathbf I)^{-1}
\mathbf B(\mathbf A+\mathbf B-\alpha\mathbf I)^{-1},
\end{equation}
 we get
\begin{equation}\label{resolv}
 |s_n(\alpha,z)-s_n^{(c)}(\alpha,z)|\le \frac1{2n}\E|\Tr \mathbf R^{(c)}
(\mathbf V(z)-\mathbf V^{(c)}(z))\mathbf J\mathbf
 R|.
\end{equation} 
Let
\begin{equation}\notag
\mathbf H^{(\nu)}=\left(\begin{matrix}&\mathbf X^{(\nu)}&\mathbf O
\\&\mathbf O&{\mathbf X^{(m-\nu+1)}}^*\end{matrix}\right)
\quad\text{and}\quad\mathbf H^{(\nu,c)}=\left(\begin{matrix}&\mathbf X^{(\nu,c)}&\mathbf O
\\&\mathbf O&{\mathbf X^{(m-\nu+1,c)}}^*\end{matrix}\right)
\end{equation}
Introduce the matrices
$$
\mathbf V_{a,b}=\prod_{q=a}^b\mathbf H^{(q)},\quad\mathbf V_{a,b}^{(c)}=\prod_{q=a}^b\mathbf H^{(q,c)}.
$$
We have
\begin{equation}\label{repr1}
\mathbf V(z)-\mathbf V^{(c)}(z)=[\mathbf V-\mathbf V^{(c)}]\mathbf J=\left[\sum_{q=1}^{m-1}\mathbf V^{(c)}_{1,q-1}
(\mathbf H^{(q)}
-\mathbf H^{(q,c)})\mathbf V_{q+1,m}\right]\mathbf J.
\end{equation}
 Applying $\max\{\|\mathbf
R\|,\,\|\mathbf R^{(c)}\|\}\le v^{-1}$, inequality (\ref{resolv}),  and  the  representations (\ref{repr1})  together, 
 we get
\begin{equation}\label{st1}
 |s_n(\alpha,z)-s_n^{(c)}(\alpha,z)|\le \frac{C}{\sqrt n}\sum_{q=1}^{m}\E ^{\frac12}
\|(\mathbf X^{(q+1)}-\mathbf X^{(q+1,c)})\|_2^2\frac{1}{\sqrt n}\E^{\frac12}\|\mathbf V^{(c)}_{1,q-1}
\mathbf R\mathbf R^{(c)}\mathbf V_{q+1,m}\|_2^2.
\end{equation}
By  multiplicative inequalities  for the matrix norm, we get
\begin{equation}\notag
 \E\|\mathbf V^{(c)}_{1,q-1}\mathbf R\mathbf R^{(c)}\mathbf V_{q+1,m}\|_2^2
\le\frac C{v^4}\E\|\mathbf V^{(c)}_{1,q-1}\mathbf V_{q+1,m}\|_2^2
\end{equation}
Applying the result of Lemma \ref{norm2}, we obtain
\begin{equation}\label{st2}
 \E\|\mathbf V_{1,q-1}^{(c)}\mathbf R\mathbf R^{(c)}\mathbf V_{q+1,m}\|_2^2\le \frac{Cn}{v^4}.
\end{equation}
Direct calculations show that
\begin{equation}\label{st3}
 \frac1n\E\|\mathbf X^{(q)}-\mathbf X^{(q,c)}\|_2^2\le \frac C{n^2}\sum_{j,k=1}^n\E|X^{(q)}_{jk}|^2
I_{\{|X^{(q)}_{jk}|\ge c \tau_n\sqrt n\}}\le CL_n(\tau_n).
\end{equation}
Inequalities (\ref{st1}), (\ref{st2}) and {\ref{st3}) together imply
\begin{equation}\label{stieltjes1}
 |s_n(\alpha,z)-s_n^{(c)}(\alpha,z)|\le\frac {C\sqrt {L_n(\tau_n)}}{v^2}.
\end{equation}

Furthermore, by definition of $X_{jk}^{(c)}$, we have
\begin{equation}\notag
|\E X_{jk}^{(q,c)}|\le \frac1{c\tau_n\sqrt n}\E |{X_{jk}^{(q)}}|^2 I_{\{|X_{jk}|\ge c\tau_n\sqrt n\}}.
\end{equation}
This implies that
\begin{equation}\label{st4}
 \|\E\mathbf X^{(q,c)}\|_2^2\le \frac C{n}\sum_{j=1}^{n}\sum_{k=1}^{n}|\E X_{jk}^{(q,c)}|^2
\le
\frac {CL_n(\tau_n)}{c\tau_n^2}.
\end{equation}
Corresponding to ${\mathbf H}^{(\nu,c)}$ introduce  ${\widetilde{\mathbf H}}^{(\nu,c)}:=\left(\begin{matrix}&\mathbf X^{(\nu,c)}-\E\mathbf X^{(\nu,c)})
&\mathbf O\\&\mathbf O&(\mathbf X^{(\nu,c)}-\E\mathbf
X^{(\nu,c)})^*\end{matrix}\right)$ and for the matrices ${\mathbf W}^{(c)},
{\mathbf V}^{(c)},{\mathbf V}^{(c)}_{a,b}$ define matrices ${\widetilde{\mathbf W}}^{(c)}$, ${\widetilde{\mathbf V}}^{(c)}$, 
${\widetilde {\mathbf V}}^{(c)}_{a,b}$ respectively. Denote  by $\widetilde{\mathcal F}_n^{(c)}(x)$  
the empirical distribution of the   squared singular values of the matrix
 $\widetilde {\mathbf V}^{(c)}(z):=\widetilde {\mathbf V}^{(c)}\mathbf J-\mathbf J(z)$. Let 
${\widetilde s}_n^{(c)}(\alpha,z)$ denote the Stieltjes transform of the distribution function 
$\widetilde F_n^{(c)}=\E\widetilde {\mathcal F}_n^{(c)}$,
\begin{equation}\notag
 {\widetilde s}_n^{(c)}(\alpha,z)=\int_{-\infty}^{\infty}\frac1{x-\alpha}d\widetilde F_n^{(c)}(x).
\end{equation}

Similar to inequality (\ref{st1}) we get
\begin{equation}\notag
 |s_n^{(c)}(\alpha,z)-\widetilde s_n^{(c)}(\alpha,z)|\le \sum_{q=0}^{m-1}
\frac1{\sqrt n}\|\E\mathbf X^{(q,c)}\|_2\frac1{\sqrt n}\E^{\frac12}\|{\widetilde {\mathbf V}_{0,q}}^{(c)}
\mathbf R^{(c)}
\widetilde {\mathbf R}^{(c)}\widetilde{\mathbf V}^{(c)}_{q+1,m}\|_2^2.
\end{equation}
Similar  to inequality (\ref{st2}), we get
\begin{equation}\notag
 \frac1{ n}\E\|\widetilde {\mathbf V}_{0,q}^{(c)}\mathbf R^{(c)}
\widetilde {\mathbf R}^{(c)}\widetilde {\mathbf V}^{(c)}_{q+1,m}\|_2^2\le \frac C{v^4}.
\end{equation}
By inequality (\ref{st4}),
\begin{equation}\notag
 \|\E \mathbf X^{(q,c)}\|_2\le \frac{C\sqrt {L_n(\tau_n)}}{c\tau_n}.
\end{equation}
The last two inequalities together imply that
\begin{equation}\label{stieltjes2}
 |s_n^{(c)}(\alpha,z)-{\widetilde s}_n^{(c)}(\alpha,z)|\le  \frac{C\sqrt{L_n(\tau_n)}}{\sqrt n\tau_nv^2}\le 
\frac{\tau_n}{\sqrt nv^2}
\end{equation}
Inequalities (\ref{stieltjes1}) and (\ref{stieltjes2}) together imply that the
matrices $\mathbf W$ and
$\widetilde {\mathbf W}^{(c)}$ have the same limit distribution.
In the what follows we may assume without loss of generality  for any $\nu=1,\ldots, m$ and $j=1,\ldots n$, 
$k=1,\ldots,n$ and any $l=1,\ldots, m$, that
\begin{equation}\label{conditions}
 \E X^{(\nu)}_{jk}=0,\quad \E {X^{(\nu)}_{jk}}^2=1, \quad\text{and}\quad |X^{(\nu)}_{jk}|\le c\tau_n\sqrt n
\end{equation}
with
\begin{equation}\notag
 L_n(\tau_n)/\tau_n^2\le\tau_n.
\end{equation}



%
%

\section{The Limit Distribution of Singular Values of the  Matrices $\mathbf V(z)$}\label{convergence}
Recall that $\mathbf H^{(\nu)}=\left(\begin{matrix}&\mathbf X^{(\nu)}& \mathbf O\\
&\mathbf O & \mathbf {X^{(m-\nu+1)}}^*\end{matrix}\right)$ and 
$\bold J(z):=\left(\begin{matrix}&\mathbf O  & z\ \mathbf I \\&\overline z\ \mathbf I &\mathbf O\end{matrix}\right)$, 
$\mathbf J:=\mathbf J(1)$. For any $1\le \nu\le \mu\le m$, put
$$
\mathbf V_{[\nu,\mu]}=\prod_{k=\nu}^{\mu}\mathbf H^{(k)},\qquad\mathbf V=\mathbf V_{[1,m]}.
$$
and
$$
\mathbf V(z):=\mathbf V\mathbf J-\mathbf J(z).
$$
We introduce the following functions
\begin{align}\label{systemdef}
 s_n(\alpha,z)&=\frac1n\sum_{j=1}^n\E[\mathbf R(\alpha,z)]_{jj}=\frac1n\sum_{j=1}^n
\E[\mathbf R(\alpha,z)]_{j+nj+n}=\frac1{2n}\sum_{j=1}^{2n}\E[\mathbf R(\alpha,z)]_{jj}\notag\\
t_n(\alpha,z)&=\frac1n\sum_{j=1}^n\E[\mathbf R(\alpha,z)]_{j+nj},\quad
u_n(\alpha,z)=\frac1n\sum_{j=1}^n\E[\mathbf R(\alpha,z)]_{jj+n}.
\end{align}

\begin{thm}\label{respect}
 If the random variables $X_{jk}^{(\nu)}$ satisfy the Lindeberg condition
 (\ref{lindeb}), the following  limits exist
$$
y=y(z,\alpha)=\lim_{n\to\infty}s_n(\alpha,z),\quad t=t(z,\alpha)=\lim_{n\to\infty}t_n(\alpha,z),
$$
and  satisfy the  equations
\begin{align}\label{system0}
 &1+wy+(-1)^{m+1}w^{m-1}y^{m+1}=0,\notag\\
&y(w-\alpha)^2+(w-\alpha)-y|z|^2=0,\notag\\
&w=\alpha+\frac{zt}y.
\end{align}
\begin{rem}
 Since the Lindeberg condition holds for i.i.d. random variables and for
 uniformly  integrable  random variables the  
conclusion of Theorem \ref{respect} holds by  Theorem \ref {main} and Theorem \ref{Lindeberg}.
\end{rem}

\end{thm}
\begin{proof}In the what follows we shall denote by $\varepsilon_n(\alpha,z)$ 
 a generic  error function such that 
$|\varepsilon_n(\alpha,z)|\le \frac{C\tau_n^q}{v^r}$ for some positive constants $C,p,r$.
 By the resolvent equality, we may write
\begin{equation}\label{b1}
 1+\alpha s_n(\alpha,z)=\frac1{2n}\E\Tr\mathbf V(z)\mathbf R(\alpha,z)=
\frac1{2n}\E\Tr \mathbf V\mathbf J\mathbf R(\alpha,z)-zt_n(\alpha,z)-\overline z u_n(\alpha,z).
\end{equation}
In the following we shall write $\mathbf R$ instead of $\mathbf R(\alpha,z)$.
Introduce the notation
\begin{equation}
 \mathcal A:=\frac1{2n}\E\Tr \mathbf V\mathbf J\mathbf R
\end{equation}
and {represent $\mathcal A$ as follows}
\begin{equation}
 \mathcal A=\frac12\mathcal A_1+\frac12\mathcal A_2,
\end{equation}
where
$$
\mathcal A_1=\frac1n\sum_{j=1}^n\E[\mathbf V\mathbf J\mathbf R]_{jj},\quad 
\mathcal A_2=\frac1n\sum_{j=1}^n\E[\mathbf V\mathbf J\mathbf R]_{j+n,j+n}.
$$
By definition of the matrix $\bold V$, we have
\begin{equation}
 \mathcal A_1=\frac1n\sum_{j,k=1}^n\E X^{(1)}_{jk}[\mathbf V_{[2,m]}\mathbf J\mathbf R]_{kj}.
\end{equation}
Let $\mathbf e_1,\ldots,\mathbf e_{2n}$ be an orthonormal basis  of $\mathbb R^{2n}$.
Note that 
\begin{align}
 \frac{\partial \mathbf V_{[2,m]}\mathbf J\mathbf R}{\partial X^{(1)}_{jk}}&=
\mathbf V_{[2,m-1]}\bold e_{k+n}\bold e_{j+n}^T\mathbf J\mathbf R\notag\\&-\mathbf V_{[2,m]}
\mathbf J\mathbf R\mathbf e_j\mathbf e_k^T\mathbf V_{[2,m]}\mathbf J\mathbf R-\mathbf V_{[2,m]}
\mathbf J\mathbf R\mathbf V_{[1,m-1]}\bold e_{k+n}\bold e_{j+n}^T\mathbf J\mathbf R.
\end{align}
Applying now the Lemmas \ref{teilor}, we obtain
\begin{equation}\label{super}
 \mathcal A_1=-\frac1n\sum_{k=1}^n\E[\mathbf V_{[2,m]}
\mathbf J\mathbf R\mathbf V_{[1,m-1]}]_{kk+n}\frac1n\sum_{j=1}^n\E[\mathbf J\mathbf R]_{j+n,j}+\varepsilon_n(z,\alpha).
\end{equation}

Introduce the  notation, for $\nu=2,\ldots, m$
\begin{equation}
 f_{\nu}=\frac1n\sum_{j=1}^n\E[\mathbf V_{[\nu,m]}
\mathbf J\mathbf R\mathbf V_{[1,m-\nu+1]}]_{jj+n}
\end{equation}
We rewrite the equality (\ref{super}) using these notations
\begin{equation}\label{super1}
 \mathcal A_1=-f_2s_n(\alpha,z)+\varepsilon_n(z,\alpha).
\end{equation}

We shall investigate the asymptotics of $f_{\nu}$ for $\nu=2,\ldots,m$.
By definition of the matrix $\mathbf V_{[\nu,m]}$, we have
\begin{equation}
 f_{\nu}=\frac1n\sum_{k,j=1}^n\E X^{(\nu)}_{jk}[\mathbf V_{[\nu+1,m]}
\mathbf J\mathbf R\mathbf V_{[1,m-\nu+1]}]_{kj+n}
\end{equation}
For {simplicity}  assume that $\nu\le m-\nu$.
Then
\begin{align}\label{a1}
 \frac{\partial\mathbf V_{[\nu+1,m]}
\mathbf J\mathbf R\mathbf V_{[1,m-\nu+1]}}{\partial X^{(\nu)}}&= \mathbf V_{[\nu+1,m-\nu]}\mathbf e_{k+n}
\mathbf e_{j+n}^T\mathbf V_{[m-\nu+2,m]}
\mathbf J\mathbf R\mathbf V_{[1,m-\nu+1]}\notag\\&+\mathbf V_{[\nu+1,m]}
\mathbf J\mathbf R\mathbf V_{[1,m-\nu]}\mathbf e_{k+n}\mathbf e_{j+n}^T\notag\\&+\mathbf V_{[\nu+1,m]}
\mathbf J\mathbf R\mathbf V_{[1,\nu-1]}\mathbf e_j\mathbf e_k^T\mathbf V_{[\nu+1,m-\nu+1]}\notag\\
&-\mathbf V_{[\nu+1,m]}
\mathbf J\mathbf R\mathbf V_{[1,\nu-1]}\mathbf e_j\mathbf e_k^T\mathbf V_{[\nu+1,m]}
\mathbf J\mathbf R\mathbf V_{[1,m-\nu+1]}\notag\\
&-\mathbf V_{[\nu+1,m]}
\mathbf J\mathbf R\mathbf V_{[1,m-\nu]}\mathbf e_{k+n}\mathbf e_{j+n}^T\mathbf V_{[m-\nu+2,m]}\mathbf J
\mathbf R\mathbf V_{[1,m-\nu+1]}
\end{align}

Applying the Lemmas \ref{teilor} again, we get
\begin{align}
 f_{\nu}&=\frac1n\sum_{k=1}^n\E[\mathbf V_{[\nu+1,m]}
\mathbf J\mathbf R\mathbf V_{[1,m-\nu]}]_{kk+n}\notag\\&-\frac1n\sum_{k=1}^n\E[\mathbf V_{[\nu+1,m]}
\mathbf J\mathbf R\mathbf V_{[1,m-\nu]}]_{kk+n}\frac1n\sum_{j=1}^n\E[\mathbf V_{[m-\nu+2,m]}\mathbf J
\mathbf R\mathbf V_{[1,m-\nu+1]}]_{j+nj+n}\notag\\&=f_{\nu+1}(1-\frac1n\sum_{j=1}^n
\E[\mathbf V_{[m-\nu+2,m]}\mathbf J\mathbf R\mathbf V_{[1,m-\nu+1]}]_{j+nj+n})
\end{align}
Note that
\begin{align}
 \frac1n\sum_{j=1}^n\E[\mathbf V_{[m-\nu+2,m]}\mathbf J\mathbf R\mathbf V_{[1,m-\nu+1]}]_{j+nj+n}=
\frac1n\sum_{j=1}^n\E[\mathbf V_{[1,m]}\mathbf J\mathbf R]_{j+nj+n}
\end{align}
Furthermore,
\begin{equation}\label{a10}
\frac1n\sum_{j=1}^n\E[\mathbf V_{[1, m]}\mathbf J\mathbf R]_{j+nj+n}=
1+\alpha s_n(\alpha,z)+\overline z u_n(\alpha,z).
\end{equation}

Relations (\ref{a1})--(\ref{a10}) together imply
\begin{equation}
 f_{\nu}=f_{\nu+1}(-\alpha s_n(\alpha,z)-\overline zu_n(\alpha,z))+\varepsilon_n(z,\alpha).
\end{equation}
By induction we get
\begin{equation}\label{super100}
 f_2=(-1)^{m-1}(\alpha s_n(\alpha,z)+\overline z u_n(\alpha,z))^{m-1}s_n(\alpha,z)+\varepsilon_n(z,\alpha).
\end{equation}
Relations (\ref{super1}) and (\ref{super100}) together imply
\begin{equation}\label{super101}
 \mathcal A_1=(-1)^{m}(\alpha s_n(\alpha,z)+\overline z u_n(\alpha,z))^{m-1}s_n^2(z,\alpha)+\varepsilon_n(z,\alpha).
\end{equation}
Similar we get that
\begin{equation}\label{super102}
g_2=(-1)^{m-1}(\alpha s_n(\alpha,z)+z t_n(\alpha,z))^{m-1}s_n(\alpha,z)+\varepsilon_n(z,\alpha). 
\end{equation}
and
\begin{equation}\label{b2}
 \mathcal A_2=(-1)^{m}(\alpha s_n(\alpha,z)+z t_n(\alpha,z))^{m-1}s_n^2(z,\alpha)+\varepsilon_n(z,\alpha). 
\end{equation}
Consider now the function $t_n(\alpha,z)$ which we may
 represent  as follows
\begin{equation}
 \alpha t_n(\alpha,z)=\frac1n\sum_{j=1}^n\E[\bold V(z)\bold R]_{j+nj}.
\end{equation}
By definition of the matrix $\mathbf H^{(1)}$, we may write
\begin{equation}
 \alpha t_n(\alpha,z)=\frac1n\sum_{j,k=1}^n\E X_{jk}^{(m)}[\bold V_{[2,m]}\bold J\bold R]_{j+nk}-
\overline z\ s_n(\alpha,z).
\end{equation}
For the derivatives of  the  matrix $\bold V_{[2,m]}\bold J\bold R $ by $X_{jk}^{(m)}$, we get
\begin{align}\label{super12}
\frac{\partial \bold V_{[2,m]}\bold J\bold R }{\partial X^{(m)}_{jk}}&=\bold V_{[2,m-1]}\bold e_j
\bold e_k^T\bold J\bold R\notag\\&-\bold V_{[2,m]}\bold J\bold R\bold e_{k+n}\bold e_{j+n}^T\bold V_{[2,m]}
\bold J\bold R-\bold V_{[2,m]}\bold J\bold R\bold V_{[1,m-1]}\bold e_j\bold e_k^T
\bold J\bold R.
\end{align}
Relation (\ref{super12}) and Lemmas (\ref{teilor})  together imply
\begin{align}
\alpha t_n(\alpha,z)&=-\frac1n\sum_{j=1}^n\E[\bold V_{[2,m]}\bold J\bold R\bold V_{[1,m-1]}]_{j+nj}
\frac1n\sum_{k=1}^n\E[\bold R]_{k+nk}
-\overline z s_n(\alpha,z)+\varepsilon_n(z,\alpha)\notag\\&
=g_2\ t_n(\alpha,z)-
\overline z\ s_n(\alpha,z)+\varepsilon_n(z,\alpha).
\end{align}
Applying equality (\ref{super102}), we obtain
\begin{equation}\label{super105}
 \alpha t_n(\alpha,z)=(-1)^{m}(\alpha s_n(\alpha,z)+\overline z u_n(\alpha,z))^{m-1}s_n(\alpha,z)t_n(\alpha,z)-
\overline z\ s_n(\alpha,z)+\varepsilon_n(z,\alpha).
\end{equation}
Analogously we obtain
\begin{equation}\label{super106}
  \alpha u_n(\alpha,z)=(-1)^{m}(\alpha s_n(\alpha,z)+z t_n(\alpha,z))^{m-1}s_n(\alpha,z)u_n(\alpha,z)-
 z\ s_n(\alpha,z)+\varepsilon_n(z,\alpha).
\end{equation}
Multiplying equation (\ref{super105}) by $z$ and equation  (\ref{super106}) by $\overline z$ and subtracting 
the second  one from the first equation, we may conclude
\begin{equation}\label{superchto}
z t_n(\alpha,z)=\overline z u_n(\alpha,z)+\varepsilon_n(z,\alpha).
\end{equation}
The last relation implies that
\begin{equation}\label{b3}
 \mathcal A_1=\mathcal A_2+\varepsilon_n(z,\alpha).
\end{equation}
Relations (\ref{b1}), \ref{super101}), (\ref{b2}), (\ref{superchto}, and (\ref{b3}) together imply
\begin{equation}\label{supernu}
 1+\alpha\ s_n(\alpha,z)=(-1)^m(\alpha s_n(\alpha,z)+z \ t_n(\alpha,z))^{m-1}s_n^2(z,\alpha)
-z \ t_n(\alpha,z)+\varepsilon_n(z,\alpha).
\end{equation}

Introduce the notations 
\begin{equation}
 y_n:=s_n(\alpha,z),\quad w_n:=\alpha+\frac{z\ t_n(\alpha,z)}{y_n}.
\end{equation}

Using these notations  we may rewrite the equations (\ref{supernu}) and (\ref{superchto}) as follows 
%
\begin{align}\label{system}
 &1+w_ny_n=(-1)^{m}y_n^{m+1}w_n^{m-1}+\varepsilon_n(z,\alpha)\notag\\
&(w_n-\alpha)+(w_n-\alpha)^2y_n-y_n|z|^2=\varepsilon_n(z,\alpha).
\end{align}
Let $n,n'\to\infty$. Consider the  difference $y_n-y_{n'}$.
From the first inequality it follows that
\begin{equation}
 |y_n-y_{n'}|\le \frac{|\varepsilon_{n,n'}(z,\alpha)|+|w_n-w_{n'}||y_n+(-1)^{m+1}y_{n'}^{m+1}
(w_n^{m-2}+\cdots+w_{n'}^{m-2})|}{|w_n+(-1)^{m+1}y_{n'}^{m+1}(w_n+(-1)^{m+1}w_{n}^{m-1}(y_n^{m}+
\cdots+y_{n'}^{m})|}
\end{equation}
Note that $\max\{|y_n|,\ |y_{n'}|\}\le \frac1v$ and $\max\{|w_n|,\ |w_{n'}|\}\le C+v$ for some positive constant
 $C=C(m)$ depending of $m$. We may choose a sufficiently large $v_0$ such that for any $v\ge v_0$ we obtain
\begin{equation}\label{100}
|y_n-y_{n'}|\le \frac{|\varepsilon_{n,n'}(z,\alpha)|}{v}+\frac Cv|w_n-w_{n'}|. 
\end{equation}
Furthermore, the second equation implies that
\begin{equation}
 (w_n-w_{n'})(1+y_n(w_n+w_{n'}-2\alpha))=(y_n-y_{n'})((w_n-\alpha)^2-|z|^2)+\varepsilon_{n,n'}(z,\alpha).
\end{equation}
It is straightforward to check that $\max\{|w_n-\alpha|, |w_{n'}-\alpha|\}\le (1+|\varepsilon_n(z,\alpha)|)|z|$. 
This implies that there exists $v_1$ such that for any $v\ge v_1$
\begin{equation}\label{101}
 |w_n-w_{n'}|\le |\varepsilon_{n,n'}(z,\alpha)|+4|z|^2|y_n-y_n'|.
\end{equation}
Inequalities (\ref{100}) and (\ref{101}) together imply that there exists a constant $V_0$ such that
for any $v\ge V_0$
\begin{equation}
 |y_n-y_n'|\le |\varepsilon_{n,n'}(\alpha,z)|,
\end{equation}
where $\varepsilon_{n,n'}(\alpha,z)\to 0$ as $n\to \infty$ uniformly with respect to $v\ge V_0$ and 
$|u|\le C$ ($\alpha=u+iv$). Since $y_n,y_{n'}$ are  locally
bounded {analytic functions in the upper half-plane}  
we may conclude by Montel's Theorem (see, for instance,
\cite{Conway}, p. 153, Theorem 2.9) that there exists an analytic  function
$y_0$   in the upper half-plane such that 
$\lim y_n=y_0$. Since $y_n$ are Nevanlinna functions, (that is analytic
functions mapping the 
upper half-plane into itself)  $y_0$ will be a Nevanlinna function too and there exists some 
distribution function $F(x,z)$ such that
$y_0=\int_{-\infty}^{\infty}\frac1{x-\alpha}dF(x,z)$ and
\begin{equation}\label{pip}
 \Delta_n(z):=\sup_x|F_n(x,z)-F(x,z)|\to 0\quad\text{as}\quad n\to\infty.
\end{equation}
The function $y_0$ satisfies the equations (\ref{system0}).

Thus Theorem \ref{respect} is proved.

\end{proof}

\section{Properties of Limit Measures}\label{property}
In this section we study the measure $F(x,z)$ with Stieltjes transform 
$s(\alpha,z)=\int_{-\infty}^{\infty}\frac1{x-\alpha}d\ F(x,z)$ satisfying 
the equations
\begin{align}\label{shift}
 &1+wy+(-1)^{m+1}w^{m-1}y^{m+1}=0,\notag\\
&y(w-\alpha)^2+(w-\alpha)-y|z|^2=0.
\end{align}

Consider the first  equation in (\ref{shift}) with  $w=u
+\sqrt{-1} \, v$. 
Assume that there are two solutions of these equation, say  $y_1$ and $y_2$,  which are Stieltjes transform of some measures.
Then we have
\begin{equation}
(y_1-y_2)w+(-1)^{m+1}w^{m-1}(y_1-y_2)(y^{m}+\cdots+y_2^m)=0.
\end{equation}
Note that
\begin{equation}
 \im\{(-1)^{m+1}w^{m-1}y_j^m\}\ge 0,\quad j=1,2.
\end{equation}
Indeed, by equation (\ref{shift})
\begin{equation}
 \im\{(-1)^{m+1}w^{m-1}y_j^m\}=\frac{\im y_j}{|y_j|^2}-v\ge v\ \left(\frac{\E|\xi-w|^{-2}}{|\E(\xi-w)^{-1}|^2}-1\right)\ge 0.
\end{equation}
Note that if $\im\xi_j^m\ge0$ for $j=1,2$ then $\im\{\xi_1^k\xi_2^{m-k}\}\ge0$ for every $k=0,\ldots,m$. This implies 
that 
\begin{equation}
 \im\{(-1)^{m+1}w^{m-1}y_1^ky_2^{m-k}\}\ge 0.
\end{equation}
From here it follows that
\begin{equation}
 |w+(-1)^{m+1}w^{m-1}(y^{m}+\cdots+y_2^m)|\ge v>0
\end{equation}
and
\begin{equation}
 y_1=y_2.
\end{equation}
It is well-known that the Stieltjes transform of a distribution function
$F(x)$  with moments given by the  Fuss--Catalan numbers 
$FC(m,p)=\frac1{mp+p}\binom{mp+p}{p}$   satisfies the equation (\ref{shift})
(see,for instance, \cite{AGT:2010}). This distribution has  bounded support 
given by  $|w|\le C_m:= 
\sqrt{\frac{(m+1)^{m+1}}{m^m}}$.

The second equation has a solution
\begin{equation}\label{shift1}
 w-\alpha=\frac{-1+\sqrt{1+4y^2|z|^2}}{2y},
\end{equation}
with $\im\{w-\alpha\}\ge0$ and $|w-\alpha|\le|z|^2$.

\begin{cor}\label{st1*}Let $p(x,z)$ denote the density of the  measure $\nu(x,z)$ with Stieltjes transform $s(\alpha,z)$.
Then, for any $|z|$ and $|x|\ge C_m+|z|$, we have
\begin{equation}
p(x,z)=0
\end{equation}
Otherwise  $p(x,z)>0$ holds.
For $z=0$ we have
\begin{equation}
 p(x,z)=O(|x|^{-\frac{m-1}{m+1}})\quad\text{as}\quad x\to0.
\end{equation}

\end{cor}
It is straightforward to check that the logarithmic potential of the measure
$\mu^{(m)}$ (the $m$-th power of the uniform 
distribution on the unit circle) satisfies
\begin{equation}
 U_{\mu^{(m)}}(z)=\begin{cases}-\log|z|,\quad &|z|\ge1\\ \frac m2(1-|z|^{\frac 2m}),&|z|\le 1\end{cases}.
\end{equation}

\begin{cor}\label{st0}For $x=0$ we have
\begin{equation}
s(0,z)=\begin{cases}0,&|z|>1\\ \sqrt{-1}\frac{\sqrt{1-|z|^{\frac2m}}}{|z|^{1-\frac1m}},&|z|\le1\end{cases}.
\end{equation}
\end{cor}
We investigate now the  connection of family of measures $\nu(\cdot,z)$ with
the distribution of $\zeta^m$, where 
$\zeta$ is  uniformly  distributed   on the unit disc in the complex plane.
 We prove the following Lemma.
\begin{lem}\label{deriv1} For $z=u+iv$ we have
\begin{equation}
\frac{\partial s(x,z)}{\partial u}=\frac{s(x,z)}{\sqrt{1+4|z|^2s^2(x,z)}}\frac{\partial s(x,z)}{\partial x}
\end{equation}

\end{lem}
\begin{proof}
Let $y=s(x,z)$. Denote by $R_i(y,w,z)$, $i=1,2$ the functions
\begin{align}
 R_1:&=R_1(y,w,z,x):=1+wy+(-1)^{m+1}w^{m-1}y^{m+1},\notag\\
R_2&:=R_2(y,w,z,x):=(w-x)^2y+(w-x)-|z|^2y.\notag
\end{align}
 Differentiating   both functions with respect to  $x$ and by $u$, we get
\begin{align}\label{dif1}
\frac{\partial y}{\partial u}&=-\frac{-2yu}{\frac{\partial R_1}{\partial w}\frac{\partial R_2}{\partial y}-
\frac{\partial R_2}{\partial w}\frac{\partial R_1}{\partial y}}\frac{\partial R_1}{\partial w}\notag\\
\frac{\partial y}{\partial x}&=-\frac{-2(w-x)y-1}{\frac{\partial R_1}{\partial w}\frac{\partial R_2}{\partial y}-
\frac{\partial R_2}{\partial w}\frac{\partial R_1}{\partial y}}\frac{\partial R_1}{\partial w}.
\end{align}
It follows immediately that
\begin{equation}
\frac{\partial y}{\partial u}=\frac{-2uy}{-2(w-x)y-1} \frac{\partial y}{\partial x}
\end{equation}

Taking in account the equality (\ref{shift1}), we get

\begin{equation}
\frac{\partial y}{\partial u}=2u\frac{y}{\sqrt{1+4|z|^2y^2}}\frac{\partial y}{\partial x},
\end{equation}
which completes the proof.
\end{proof}
Introduce now the function
$$
V(z)=-\int_{-\infty}^{\infty}\log|x|d\nu(z,x).
$$

\begin{lem}The following relation holds
$$
V(z)=U_{\mu^{(m)}}(z).
 $$
\end{lem}

\begin{proof}
We start from the simple equality, for $z=u+iv$,
$$
\frac{\partial U_{\mu^{(m)}}(z)}{\partial u} =\begin{cases}-\frac u{u^2+v^2}, &|z|\ge1\\
-\frac u{{(u^2+v^2)}^{\frac{m-1}{m}}}, &|z|<1\end{cases}.
$$
We prove that
$$
\frac{\partial V(z)}{\partial u}=\frac{\partial U_{\mu}(z)}{\partial u}.
$$
Let  $\Delta(x)=-\sqrt{-1}s(z,\sqrt{-1}x)$, where $x>0$. The symmetry of function $\nu(z,y)$ in $y$ implies 
that the function
$\Delta(x)$ will be real and non-negative.
We have
\begin{equation}\label{rep100}
\Delta(x)=\int_{-\infty}^{\infty}\frac {x}{x^2+y^2}d\nu(z,y).
\end{equation}
By Corollary \ref{st0}, we have
\begin{equation}\lim_{x\to0}\Delta(x)=\begin{cases}0, &|z|>1\\ \frac{\sqrt{1-|z|^{\frac2m}}}{|z|^{1-\frac1m}},&|z|
\le1\end{cases}
\end{equation}
Note that $\lim_{x\to\infty}\Delta(x)=0$. We consider integral
$$
B(C,z)=\int_0^C\Delta(x)dx.
$$
Using the representation (\ref{rep100}), we get
\begin{equation}
B(C,z)=-\int_{-\infty}^{\infty}\log|y|p(y,z)dy
+\frac12\int_{-\infty}^{\infty}\log(1+\frac{y^2}{C^2})p(y,z)dy+\log
C.
\end{equation}
We rewrite this equality as follows
\begin{equation}
 V(z)=B(C,z)+\frac12\int_{-\infty}^{\infty}\log(1+\frac{y^2}{C^2})p(y,z)dy+\log
C,
\end{equation}
which implies
\begin{equation}\label{deriv2}
\frac{\partial}{\partial u}V(z)=\frac{\partial}{\partial u}B(C,z)+\frac12\frac{\partial}{\partial u}
\int_{-\infty}^{\infty}\log(1+\frac{y^2}{C^2})p(y,z)dy.
\end{equation}

According to Lemma \ref{deriv1}, we get
\begin{equation}
\frac{\partial \Delta(x)}{\partial
u}=\frac{2u\Delta(x)}{\sqrt{1-4|z|^2\Delta^2(x)}}\frac{\partial
\Delta(x)}{\partial x},
\end{equation}
Note that the quantity $\Delta(x)$ satisfies  $0\le \Delta(x)\le \frac1{2|z|}$.  There exists a point $x_0$ such that 
$\Delta(x_0)=\frac1{2|z|}$. Thus  we
get
\begin{align}
\frac{\partial}{\partial
u}&\int_0^{C}\Delta(x)dx=\int_0^C\frac{\partial}{\partial
u}\Delta(x)dx=2u\int_0^C\frac{\Delta(x)}{\sqrt{1-4|z|^2\Delta^2(x)}}
\frac{\partial}{\partial x}\Delta(x)dx\\
&=u\left(\int_{\Delta(0)}^{\frac1{2|z|^2}}
+\int_{\Delta(C)}^{\frac1{2|z|^2}}
\right)\frac{d(a^2)}{\sqrt{1-4a^2|z|^2}}\notag\\&=\frac{-u}{2|z|^2}\left(\sqrt{1-4|z|^2\Delta^2(C)}+
\sqrt{1-4|z|^2\Delta^2(0)}\right)
\end{align}
Simple calculations show that in the limit  $C\to\infty$, we obtain
\begin{equation}\label{deriv3}
\lim_{C\to\infty}\frac{\partial}{\partial u}B(C,z)=\lim_{C\to\infty}\frac{\partial}{\partial u}\int_0^C\Delta(x)dx=
\begin{cases}{-\frac u{|z|^2},\quad \text{  if  }|z\ge1}\\
{-\frac{u}{|z|^{2-\frac2m}},\quad\text{  if  }|z|>1}\end{cases}.
\end{equation}
Consider now the quantity
$$
A(C)=\frac{\partial}{\partial
u}\int_{-\infty}^{\infty}\log\{1+\frac{y^2}{C^2}\}p(y,z)dy.
$$
By Corollary \ref{st1}, we have
\begin{equation}\label{rest}
A(C)=\frac{\partial}{\partial
u}\int_{-x_3}^{x_3}\log\left(1+\frac{y^2}{C^2}\right)p(y,z)dy.
\end{equation}
Using equality $p(y,z)=\im s(z,y)$, we may rewrite equality (\ref{rest}) as follows
\begin{equation}
 A(C)=\im\left\{\int_{-C_0}^{C_0}\log\left(1+\frac{y^2}{C^2}\right)\frac{\partial}{\partial
u}s(z,y)dy\right\}.
\end{equation}
Applying Lemma \ref{deriv1}, we get
\begin{equation}
 A(C)=\im\left\{\int_{-2C_0}^{2C_0}\log\left(1+\frac{y^2}{C^2}\right)
\frac{s(y,z)}{\sqrt{1+4|z|^2s^2(y,z)}}\frac{\partial s(y,z)}{\partial y}dy\right\}.
\end{equation}
Integrating by parts and using the inequality $|\log(1+\frac{y^2}{C^2})|\le \frac{\gamma y^2}{C^2}$
with some constant $\gamma>0$, and $|s(2C_0,z)|\le\frac1{C_0}$, and $|s(0,z)|\le \frac1{2|z|}$, we conclude that
\begin{equation}\label{deriv4}
 \lim_{C\to\infty}A(C)=0.
\end{equation}
Collecting the relations (\ref{deriv2}), (\ref{deriv3}), and (\ref{deriv4})
concludes the proof of the  Lemma.
\end{proof}


%

\section{The Minimal Singular Value of the Matrix $\bold W-z\bold I$}\label{singular}
Recall that 
$$
\bold W=\prod_{\nu=1}^m\bold X^{(\nu)},
$$
where $\bold X^{(1)},\ldots,\bold X^{(m)}$ are independent $n\times n$ matrices with independent entries.
Let $\bold W(z)=\bold W-z\bold I$ and let $s_n(\bold A)$ denote
the minimal singular value of a matrix $\bold A$.
Note that
$$
s_n(\bold A)=\inf_{\bold x:\|\bold x\|=1}\|\bold A \bold x\|_2.
$$
Introduce the matrix $\bold W^{(1)}=\prod_{\nu=2}^m\bold X^{(\nu)}$.
We may write
\begin{equation}
 s_n(\bold W(z))=\inf_{\bold x:\|\bold x\|=1}\|\bold W(z)\bold x\|_2\ge
\inf_{\bold x:\|\bold x\|=1}\|(\bold X^{(1)}-z(\bold W^{(1)})^{-1})\bold x\|_2\inf_{\bold x:\|\bold x\|=1}
\|\bold W^{(1)}\bold x\|_2.
\end{equation}
By induction, we obtain
\begin{equation}\label{minim}
 s_n(\bold W(z))\ge s_n(\bold X^{(1)}-z(\bold W^{(1)})^{-1}))\prod_{\nu=2}^ms_n(\bold X^{(\nu)}).
\end{equation}
\begin{lem}
 \label{thm1} Let $X_{jk}^{(\nu)}$ be independent  complex random variables with $\E X_{jk}=0$ and $\E|X_{jk}|^2=1$, 
which are  
uniformly integrable , i.e.
\begin{equation}\label{unif0}
\max_{j,k,\nu}\E|X_{jk}^{(\nu)}|^2I_{\{|X_{jk}|>M\}}\to0\quad\text{as}\quad M\to \infty.
\end{equation}
Let $K\ge 1$. Then there exist constants $c, C, B >0$ depending on $\theta$  and $K$ such that for any $z\in\mathbb C$ 
and positive $\varepsilon$ we have
\begin{align}
\Pr\{s_n\le\varepsilon/n^B ;\ \max_{1\le\nu\le m}s_1(\bold X^{(\nu)})\le Kn\}\le
\exp\{ - c\, \, n \}+\frac{C\sqrt{\ln n}}{\sqrt{n}},
\end{align}
where $s_n=s_n(\bold W(z))$.
\end{lem}
\begin{proof}
 The proof is similar to the proof of Theorem 4.1 in \cite{GT:2010}.
Applying inequality (\ref{minim}), we get
\begin{align}\label{inequal1}
 \Pr\{s_n&\le\varepsilon/n^B ;\ \max_{1\le\nu\le m}s_1(\bold X^{(\nu)})\le Kn\}\notag\\&
\le
\Pr\{s_n(\bold X^{(1)}-z(\bold W^{(1)})^{-1})\le\varepsilon^{\frac1m} n^{-\frac Bm};\ \max_{1\le
\nu\le m}s_1(\bold X^{(\nu)})\le Kn\}\notag\\&\qquad\qquad\quad+\sum_{\nu=2}^m\Pr\{s_n(\bold X^{(\nu)})\le 
\varepsilon^{\frac1m}n^{-\frac Bm};\ \max_{1\le\nu\le m}s_1(\bold X^{(\nu)})\le Kn\}.
\end{align}
Furthermore,
\begin{align}\label{inequal2}
 \Pr\{s_n&(\bold X^{(1)}-z(\bold W^{(1)})^{-1})\le\varepsilon^{\frac1m} n^{-\frac Bm};\ 
\max_{1\le\nu\le m}s_1(\bold X^{(\nu)})\le Kn\}\notag\\&\le\Pr\{s_n(\bold X^{(1)}-z(\bold W^{(1)})^{-1})
\le\varepsilon^{\frac1m} n^{-\frac Bm};\ s_1(\bold X^{(\nu)})\le Kn;\ s_1({\bold W^{(1)}}^{-1})\le n^B\}
\notag\\&\qquad\qquad\quad\qquad\qquad\quad\qquad\ +\Pr\{s_1({\bold W^{(1)}}^{-1})\ge n^B;\ 
\max_{1\le\nu\le m}s_1(\bold X^{(\nu)})\le Kn\}.
\end{align}
Note that
\begin{equation}
 s_1({\bold W^{(1)}}^{-1})\le \prod_{\nu=2}^ms_1({\bold X^{(\nu)}}^{-1})=
\prod_{\nu=2}^ms_n^{-1}(\bold X^{(\nu)}).
\end{equation}
Applying this inequality and Theorem 4.1 in \cite{GT:2010}, we obtain
\begin{equation}\label{inequal3}
 \Pr\{s_1({\bold W^{(1)}}^{-1})\ge n^B;\ \max_{1\le\nu\le m}s_1(\bold X^{(\nu)})\le Kn\}\le 
\exp\{ - c\, \, n \}+\frac{C\sqrt{\ln n}}{\sqrt{n}}.
\end{equation}
with some positive constants $C,c>0$.
Moreover, adapting the proof of Theorem 4.1 in \cite{GT:2010}, we  see that
this theorem holds for  all matrices  $\bold X^{(1)}-z\bold B$  uniformly for all non-random matrices $\bold B$ such that
$\|\bold B\|_2\le Cn^{Q}$ for  some positive constant $Q>0$, i.e. 
\begin{equation}
 \Pr\{s_n(\bold X^{(1)}-z\bold B)\le \varepsilon n^{-B}, s_1(\bold X^{(1)})\le Kn\}\le 
\exp\{ - c\, \, n \}+\frac{C\sqrt{\ln n}}{\sqrt{n}}.
\end{equation}
with a constant depending on $C$ and $Q$ and not depending on the matrix $B$.
Since the  matrices $\bold X^{(1)}$ and $\bold W^{(1)}$ are independent,
 we may apply this result and get
\begin{equation}\label{inequal4}
 \Pr\{s_n(\bold X^{(1)}-z{\bold W^{(1)}}^{-1})\le \varepsilon n^{-B}, s_1(\bold X^{(1)})\le Kn;
s_1({\bold W^{(1)}}^{-1})\le Cn^B\}\le \exp\{ - c\, \, n \}+\frac{C\sqrt{\ln n}}{\sqrt{n}}
\end{equation}
Collecting the inequalities (\ref{inequal1})--(\ref{inequal4}), we conclude
the proof of the Lemma.

\end{proof}
Following Tao and Vu \cite{TV:2010}, we may prove  sharper results about the
behavior of small  singular values of a matrix product.

We shall use the following well-known fact.
Let $\bold A$ and $\mathbf B$ be $n\times n$ denote  matrices and let $s_1(\mathbf A)\ge\cdots \ge s_n(\mathbf A)$ resp.
($s_1(\mathbf B)\ge\cdots\ge s_n(\mathbf B)$ and
$s_1(\mathbf A\mathbf B)\ge\cdots \ge s_n(\mathbf A\mathbf B)$) denote 
 the singular value of a matrix
$\mathbf A$ (and the matrices  $\mathbf B$ and  $\mathbf A\mathbf B$ respectively).
Then for any $1\le k\le n$ we have
\begin{equation}\label{prod1}
 \prod_{j=k}^ns_j(\mathbf A\mathbf B)\ge \prod_{j=k}^ns_j(\mathbf A)s_j(\mathbf B),
\end{equation}
and 
\begin{equation}
\prod_{j=1}^ns_j(\mathbf A\mathbf B)= \prod_{j=1}^ns_j(\mathbf A)s_j(\mathbf B)
\end{equation}
(see, for instance \cite{HJ:91}, p.171, Theorem 3.3.4).

We need to prove a bound similar to the bound (45) in \cite{TV:2010}, namely:
\begin{equation}
 \lim_{n\to\infty}\frac1n\sum_{j=n-n\delta_n}^{n-n^{\gamma}}\ln s_j(\mathbf W-z\mathbf I)=0,
\end{equation}
for any sequence $\delta_n\to0$.
To prove this bound it is enough to prove that for any $\nu=1,\ldots,m$ and any
fixed {sequence of matrices}
$\mathbf M_n$ with $\|\mathbf M_n\|_2\le Cn^{B}$ for some positive constant $B>0$
\begin{equation}
\lim_{n\to\infty}\frac1n\sum_{j=n-n\delta_n}^{n-n^{\gamma}}\ln s_j(\mathbf X^{(\nu)}+\mathbf M_n)=0. 
\end{equation}
Indeed, it follows from (\ref{prod1}), that
\begin{equation}\label{prod2}
 \frac1n\sum_{j=n-n\delta_n}^{n-n^{\gamma}}\ln s_j(\mathbf W-z\mathbf I)\ge
\frac1n\sum_{\nu=1}^{m-1}\sum_{j=n-n\delta_n}^{n-n^{\gamma}}\ln s_j(\mathbf X^{(\nu)})
+\frac1n\sum_{j=n-n\delta_n}^{n-n^{\gamma}}\ln s_j(\mathbf X^{(m)}+\mathbf M_n),
\end{equation}
where $\mathbf M_n^{-1}=\prod_{\nu=1}^{m-1}\mathbf X^{(\nu)}$.
Note that the matrices $X^{(m)}$ and $\mathbf M_n$ are independent and it
follows from our results in  \cite{GT:2010}, Lemma A1,  that
$\|\mathbf M_n\|_2\le Cn^B$ for some $B>0$ with probability close to one.
The relations
\begin{align}
 \lim_{n\to\infty}\frac1n\sum_{j=n-n\delta_n}^{n-n^{\gamma}}\ln s_j(\mathbf X^{(\nu)})&=0,\quad\text{for}
\quad \nu=1,\ldots,m-1,\notag\\
\lim_{n\to\infty}\frac1n\sum_{j=n-n\delta_n}^{n-n^{\gamma}}\ln s_j(\mathbf X^{(m)}+\mathbf M_n)&=0
\end{align}
follow from the bound
\begin{equation}\label{prod3}
 s_j(\mathbf X^{(\nu)}+\mathbf M_n)\ge c\sqrt{\frac{n-j}{n}},\quad 1\le j\le n-n^{\gamma}.
\end{equation}
To prove this we need  the following simple Lemma.
\begin{lem}\label{vot}
 Let $\lim_{n\to\infty}\delta_n=0$ and let  $s_j$, for $n-n\delta_n\le j\le
 n-n^{\gamma}$ with $0<\gamma<1$  denote numbers 
 satisfying  the inequality
\begin{equation}
 s_j\ge c\sqrt{\frac{n-j}n}.
\end{equation}
Then
\begin{equation}
 \lim_{n\to\infty}\frac1n\sum_{n-n\delta_n\le j\le n-n^{\gamma}}\ln s_j=0.
\end{equation}

\end{lem}
\begin{proof}
 Without loss of generality we may assume that $0<s_j\le 1$.
By the conditions of {Lemma \ref{vot}}, we have
\begin{equation}
 0\ge\frac1n\sum_{n-n\delta_n\le j\le n-n^{\gamma}}\ln s_j\ge\frac1n\sum_{n-n\delta_n\le j\le n-n^{\gamma}}
\ln \{\frac{n-j}n\}=A.
\end{equation}
After summation  and using Stirling's formula, we get
\begin{align}
 |A|&\le \frac1n\ln\{\frac{[n-n\delta_n]!}{[n-n^{\gamma}]!n^{n\delta_n-n^{\gamma}}}\}\notag\\&\le
\delta_n|\ln\delta_n|+(1-\gamma)n^{\gamma-1}\ln n\to 0, \quad\text{as}\quad n\to\infty.
\end{align}
This  proves  Lemma \ref{vot}.

\end{proof}

It remains to prove inequality (\ref{prod3}).
 This result was proved by Tao and Vu in \cite{TV:2010}
(see inequality (8.4) in \cite{TV:2010}).
It represents the  crucial result in their proof of the 
circular law assuming a  second moment only.
For completeness  we repeat this proof here. We start from the following
\begin{prop}\label{Tao}
 Let $1\le d\le n-n^{\gamma}$ with $\frac{8}{15}<\gamma<1$.
and $0<c<1$, and $\mathbb H$ be a (deterministic) $d$-dimensional subspace of $\mathbb C^n$. Let $X$ be a row of
 $\mathbf A_n:=\mathbf X+\mathbf M_n$. Then 
\begin{equation}\label{dist100}
 \Pr\{\text{\rm dist}(X,\mathbb H)\le c\sqrt{n-d}\}=O(\exp\{-n^{\frac{\gamma}8}\}),
\end{equation}
where $\text{\rm dist}(X,\mathbb H)$ denotes the Euclidean distance between
a vector $X$ and a subspace
$\mathbb H$ in $\mathbb C^n$. 
\end{prop}
\begin{proof} It was proved by Tao and Vu in \cite{TV:2010} (see Proposition 5.1).
Here we sketch their proof. As shown in \cite{TV:2010} we may reduce the problem to the case that $\E X=0$.
For this it is enough to consider vectors $X'$ and $v$  such that $X=X'+v$ and $\E X'=0$.
Instead of the subspace $\mathbb H$ we may consider subspace $\mathbb H'=\text{\rm span}(\mathbb H,v)$ and note that
\begin{equation}
 \text{\rm dist}(X,\mathbb H)\ge \text{\rm dist}(X',\mathbb H').
\end{equation}
The claim follows now from a  corresponding result for random vectors with mean zero.
In what follows we assume that $\E X=0$. 
We reduce the problem to vectors with  bounded coordinates.
Let
$\xi_j=I\{|X_j|\ge n^{\frac{1-\gamma}2}\}$, where $X_j$ denotes the $j$-th
coordinate of a vector $X$.
Note that $p_n:=\E\xi_j\le n^{-(1-\gamma)}$.
Applying Chebyshev's  inequality, we get, for any $h>0$
\begin{equation}
 \Pr\{\sum_{j=1}^n\xi_j\ge 2{n^{\gamma}}\}\le \exp\{-hn^{\gamma}\}\exp\{np_n(\text{\rm e}^h-1-h)\}.
\end{equation}

Choosing $h=\frac14$, we obtain
\begin{equation}\label{prod7}
 \Pr\{\sum_{j=1}^n\xi_j\ge 2{n^{\gamma}}\}\le\exp\{-\frac{n^{\gamma}}8\}.
\end{equation}
Let $J\subset\{1,\ldots,n\}$ and $E_J:=\{\prod_{j\in J}(1-\xi_j)\prod_{j\notin J}\xi_j=1\}$.
Inequality (\ref{prod7}) implies
\begin{equation}
 \Pr\{\bigcup_{J:|J|\ge n-2n^{\gamma}}E_J\}\ge 1-\exp\{-\frac{n^{\gamma}}8\}.
\end{equation}
Let $J$ with $|J|\ge n-2n^{\gamma}$ be fixed. Without loss of generality we may assume that $J=1,\ldots,n'$ 
with some $n-2n^{\gamma}\le n'\le n$. It is now suffices to prove that
\begin{equation}\label{dist1}
 \Pr\{\text{\rm dist}(X,\mathbb H)\le c\sqrt{n-d}| E_J\}=O(\exp\{-\frac {n^{\gamma}}8\}).
\end{equation}
Let $\pi$ denote the orthogonal projection $\pi:\mathbb C^n\rightarrow\mathbb C^{n'}$. We note that
\begin{equation}\label{dist9}
 \text{\rm dist}(X,\mathbb H)\ge\text{\rm dist}(\pi(X),\pi(\mathbb H)).
\end{equation}

Let $\widetilde x$ be a random variable $x$ conditioned on 
 the event $|x|\le n^{1-\gamma}$ and let
$\widetilde X=(\widetilde x_1,\ldots, \widetilde x_n)$. The relation (\ref{dist1}) will follow now from 
\begin{equation}
 \Pr\{\text{\rm dist}(\widetilde X',\mathbb H')\le c\sqrt{n-d}\,\big| |x_j|\le n^{1-\gamma}, j\notin J\}
=O(\exp\{-\frac {n^{\gamma}}8\}),
\end{equation}
where $\mathbb H'=\pi(\mathbb H)$ and $\widetilde X'=\pi(\widetilde X)$.
We may represent the vector $\widetilde X$as  $\widetilde X=\widetilde X'+v$, where $v=\E\widetilde X$ and $\E\widetilde X'=0$. 
We reduce the claim to the bound
\begin{equation}\label{dist5}
 \Pr\{\text{\rm dist}(\widetilde X',\mathbb H'')\le c\sqrt{n-d}\,\big| |x_j|\le n^{1-\gamma}, j\notin J\}=
O(\exp\{-\frac {n^{\gamma}}8\}),
\end{equation}
where $\mathbb H''=\text{\rm span}(v,\mathbb H')$.
In the what follows we shall omit the  symbol $'$ in the notations.
 To prove (\ref{dist5}) we shall apply the following result of Maurey.
Let $\mathbb X$ denote a normed space and $f$ denote a convex function on $\mathbb X$.
Define the functional $Q$ as follows
\begin{equation}
 Qf(x):=\inf_{y\in \mathbb X}[f(y)+\frac{\|x-y\|^2}4].
\end{equation}
\begin{defn}
 We say that a measure $\mu$ satisfies the convex property $(\tau)$ if for any convex function $f$ on $\mathbb X$
\begin{equation}
 \int_{\mathbb X}\exp\{Qf\}d\mu\int_{\mathbb X}\exp\{-f\}d\mu\le 1.
\end{equation}
 
\end{defn}

We reformulate the following result of Maurey (see \cite{Maurey:1991}, Theorem 3)

\begin{thm}\label{maurey}Let $(\mathbb X_i)$ be a family of normed spaces; for each $i$, 
let $\mu_i$ be a probability measure with diameter $\le 1$ on $\mathbb X_i$,  
for $x\in\mathbb X_i$. If $\mu$ is the product of a family $(\mu_i)$, then $\mu$ satisfies the convex property $(\tau)$.

\end{thm}

As corollary of Theorem \ref{maurey} we get
\begin{cor}\label{talagrand}
  Let $\mu_i$ be a probability measure with diameter $\le 1$ on $\mathbb X$, $i=1,\ldots,n$. Let $g$ denote a 
convex $1$-Lipshitz function on $\mathbb X^n$. Let $M(g)$ denote a median of $g$.
If $\mu$ is the product of the family $(\mu_i)$, then 
\begin{equation}
 \mu\{|g-M(g)|\ge h\}\le 4\exp\{-\frac {h^2}4\}.
\end{equation}

\end{cor}

Applying Corollary \ref{talagrand} to $\mu_i$, being  the distribution of $\widetilde x_i$, we get
\begin{equation}\label{dist10}
 \Pr\left\{|\text{\rm dist}(\widetilde X,\mathbb H)-M(\text{\rm dist}(\widetilde X,\mathbb H))|
\ge rn^{\frac{1-\gamma}2}\right\}\le 4\exp\{-r^2/16\}.
\end{equation}
The last inequality implies that there exists a constant $C>0$ such that
\begin{equation}\label{dist13}
 |\E\text{\rm dist}(\widetilde X,\mathbb H)-M(\text{\rm dist}(\widetilde X,\mathbb H))|\le Cn^{\frac{1-\gamma}2},
\end{equation}
and
\begin{equation}\label{dist14}
 \E\text{\rm dist}(\widetilde X,\mathbb H)\ge \sqrt{\E(\text{\rm dist}(\widetilde X,\mathbb H))^2}-
Cn^{\frac{1-\gamma}2}.
\end{equation}
By Lemma 5.3 in \cite{TV:2010}
\begin{equation}\label{dist15}
 \E(\text{\rm dist}(\widetilde X,\mathbb H))^2=(1-o(1))(n-d).
\end{equation}
Since $n-d\ge n^{\gamma}$ the inequalities (\ref{dist13}), (\ref{dist14}) and (\ref{dist15}) together imply 
(\ref{dist100}). Thus   Proposition \ref{Tao} is proved.
\end{proof}
Now we prove (\ref{prod3}). We repeat the proof of Tao and Vu \cite{TV:2010}, inequality (8.4).
Fix $j$. Let $\mathbf A_n=\mathbf X^{(m)}-z\mathbf M_n$ and let $\mathbf A_n'$
denote a matrix formed  by the  first $n-k$ rows of 
$\mathbf A_n$ with $k=j/2$. Let $\sigma_l'$, $1\le l\le n-k$,  be singular
values of $\mathbf A_n'$ (in decreasing order). By the interlacing property
and re-normalizing we get 
\begin{equation}
 \sigma_{n-j}\ge \frac1{\sqrt n}\sigma_{n-j}'.
\end{equation}
By Lemma A.4 in \cite{TV:2010}
\begin{equation}
 T:={\sigma'_1}^{-2}+\cdots+{\sigma'_{n-k}}^{-2}={\text{\rm dist}}_1^{-2}+\cdots+{\text{\rm dist}}^{-2}_{n-k}.
\end{equation}
Note that
\begin{equation}
 T\ge (j-k){\sigma'}_{n-j}^{-2}=\frac j2{\sigma'}_{n-j}^{-2}.
\end{equation}
Applying Proposition \ref{Tao}, we get that with probability $1-\exp\{-n^{\gamma}\}$
\begin{equation}
 T\le \frac{n}j.
\end{equation}
Combining the last inequalities, we get (\ref{prod3}).
\begin{lem}\label{spec}
 Under the conditions of Theorem \ref{main} there exists a constant $C$ such that for any   $k\le n(1-
C\Delta_n^{\frac1{m+1}}(z))$,
\begin{equation}
 \Pr\{s_k\le \Delta_n(z)\}\le C\Delta_n^{\frac1{m+1}}(z).
\end{equation}

\end{lem}
\begin{proof}
 Recall that $F_n(x,z)=\E\mathcal F_n(x,z)$ denotes the mean of the spectral distribution function $\mathcal F_n(x,z)$  
of the matrix
$\mathbf H(z)$ and that $F(x,z)=\lim_{n\to\infty}F_n(x,z)$. According to Theorem \ref{respect}, the Stieltjes transform 
of the distribution function $F_n(x,z)$ satisfies the system of algebraic equations (\ref{system0}) and 
\begin{equation}
 \Delta_n(z)=\sup_x|F_n(x,z)-F(x,z)|\to 0\quad\text{as}\quad n\to\infty.
\end{equation}
We may write, for any $k=1,\ldots,n$,
\begin{equation}
 \Pr\{s_k\le \Delta_n(z)\}\le \Pr\{\mathcal F_n(s_k,z)\le\mathcal F_n(\Delta_n(z)\}\le
 \Pr\{\frac{n-k}n\le \mathcal F_n(\Delta_n(z)\}.
\end{equation}
Applying  Chebyshev's  inequality, we obtain
\begin{equation}
 \Pr\{s_k\le \Delta_n(z)\}\le \frac{n\E\mathcal F_n(\Delta_n(z))}{n-k}\le
\frac{n(F(\Delta_n(z),z)+\Delta_n(z)}{n-k}.
\end{equation}
It is straightforward to check that from the system of equations  (\ref{system0}) it follows
\begin{equation}
 F(\Delta_n(z),z)\le C\Delta_n^{\frac2{m+1}}(z).
\end{equation}
The last inequality concludes the proof of Lemma \ref{spec}.

\end{proof}
\begin{lem}\label{compact}
Let $\Delta_n(z):=\sup_x|F_n(x,z)-F(x,z)|$. Then there exists some absolute positive constant $R$ such that
\begin{equation}\label{rate}
 \Pr\{|\lambda_{k_1}|>R\}\le C\sqrt{\Delta_n(z)},
\end{equation}
where $k_1:=\big[\Delta_n^{\frac14}(z)n\big]$.
\end{lem}
\begin{proof}It is straightforward to check from (\ref{system}) that the 
distribution $F(x,z)$ is compactly  supported. 
Fix $R$ such that $F(R,z)=1$. 
Let us introduce $k_0:=\big[\Delta_n^{\frac12}n\big]$. Using Chebyshev's inequality we obtain, for  $R>0$,
\begin{equation}\notag
 \Pr\{s_{k_0}>R\}\le\frac{1-\E F_n(R)}{k_0/n}\le \Delta_n^{\frac12}.
\end{equation}
On the other hand, 
\begin{equation}\notag
 \Pr\{|\lambda_{k_1}|>R\}\le\Pr\{\prod_{\nu=1}^{k_1}|\lambda_{\nu}|>R^{k_1}\}\le\Pr\{\prod_{\nu=1}^{k_1}
s_{\nu}>R^{k_1}\}\le\Pr\{\frac1{k_1}\sum_{\nu=1}^{k_1}\ln{s_{\nu}^{(m)}}>\ln{R}\}.
\end{equation}
Let $k_2=\max\{1\le j\le k_0:\ \sigma_j\ge\Delta_n^{-1}(z)\}$. If $\sigma_1\le \Delta_n^{-1}(z)$ then 
$k_2=0$.
Furthermore, for any value $R_1\ge 1$, splitting into the events $s_{k_0}>R$ and
$s_{k_0}\le R$, we get
\begin{align}
\Pr\{\frac1{k_1}\sum_{\nu=1}^{k_1}\ln{s_{\nu}}>\ln{R_1}\}&\le\Pr\{s_{k_0}>R\}\notag\\&+\Pr\{\frac{1}{k_1}
\sum_{j=k_2+1}^{k_0}\ln{s_j}+
\ln{R}>\frac12\ln{R_1}\}
+\Pr\{\frac1{k_1}\sum_{j=1}^{k_2}
\ln s_j>\frac12\ln{R_1}\}
\end{align}
Applying Chebyshev's  inequality, we get
\begin{align}\Pr\{\frac1{k_1}\sum_{\nu=1}^{k_1}\ln{s_{\nu}}>\ln{R_1}\}&\le\Pr\{s_{k_0}>R\}\notag\\&
+\Pr\{\frac{k_0}{k_1}\ln {\Delta_n^{-1}(z)}>\frac12\ln{\frac{R_1}{R^2}}\}
+\frac{n}{k_1}\int_{\Delta_n^{-1}(z)}\ln xdF_n(x,z).\notag
\end{align}
Now choose $R_1:=2R^2$. Thus, since $k_1/k_0\sim \Delta_n^{\frac14}(z)$, and $\Delta_n^{\frac14}(z)|
\ln\Delta_n(z)|\to0$, we get for sufficiently large $n$
\begin{equation}\notag
 \Pr\{|\lambda_{k_1}|>R\}\le\Delta_n^{\frac{1}{2}}+\frac{n}{k_1}\int_{\Delta_n^{-1}(z)}\ln xd\ F_n(x,z).
\end{equation}
Taking into account that the function $\frac{\ln x}{x^2}$ decreases in the interval
 $[\delta_n^{-1}(z),\infty)$, we get 
\begin{equation}
 \frac{n}{k_1}\int_{\Delta_n^{-1}(z)}^{\infty}\ln x d\ F_n(x,z)\le\frac{n\Delta_n^2(z)}{k_1}\ln{\Delta_n^{-1}(z)}
\int_0^{\infty} x^2d\ F_n(x,z)\le \Delta_n^{\frac12}(z)\ln{\Delta_n^{-1}(z)}.
\end{equation}
 Thus  the Lemma is proved.
\end{proof}

\section{Proof of the  Main Theorem}\label{proof}
In this Section we give the proof of Theorem \ref{main}. 
 For any $z\in \mathbb C$ and an absolute constant $c>0$ we introduce the set
$\Omega_n(z)=\{\omega\in\Omega:\ c/n^B\le s_n(z), \ s_1\le n,
\ |\lambda_{k_1}|\le R\ s_{k_2}\ge \Delta_n(z)\}$.
According to Lemma \ref{largeval}
\begin{equation}
\Pr\{s_1(\bold X)\ge n\}\le Cn^{-1}.\notag
\end{equation}

Due to Lemma \ref{thm1} with $\varepsilon=c$, we have
\begin{equation}
\Pr\{c/n^B\ge s_n(z)\}\le\frac{C\sqrt{\ln n}}{\sqrt {n}}+\Pr\{s_1\ge n\}.\notag
\end{equation}
According to Lemma \ref{compact}, we have
\begin{equation}
 \Pr\{|\lambda_{k_1}|\le R\}\le C\sqrt{\Delta_n}.
\end{equation}
Furthermore, in view of  Lemma \ref{spec},
\begin{equation}
\Pr\{s_k\le \Delta_n(z)\}\le C\Delta_n^{\frac1{m+1}}(z).
\end{equation}
These inequalities imply
\begin{equation}\label{truncation}
\Pr\{\Omega_n(z)^{c}\}\le C\Delta_n^{\frac1{m+1}}(z).
\end{equation}
The remaining part  of the proof of Theorem \ref{main} is similar to the 
proof of Theorem 1.1 in the paper of 
G\"otze and Tikhomirov \cite{GT:2010}. For completeness we shall repeat it here.
Let $r=r(n)$ be such that $r(n)\to 0$ as $n\to\infty$. 
A more specific choice will be made  later.
Consider the potential $U_{\mu_n}^{(r)}$.
We have
\begin{align}
U_{\mu_n}^{(r)}&=-\frac1n\E\log|\det(\bold W-z\bold I-r\xi\bold I)|\notag\\&
=-\frac1n\sum_{j=1}^n\E\log|\lambda_j^{(m)}-r\xi-z|I_{\Omega_n(z)}
-\frac1n\sum_{j=1}^n\E\log|\lambda_j^{(m)}-r\xi-z|I_{\Omega_n^{(c)}(z)}\notag\\&=
\overline U_{\mu_n}^{(r)}+\widehat U_{\mu_n}^{(r)},\notag
\end{align}
where
$I_A$ denotes an indicator function of an event $A$ and ${\Omega_n(z)}^{c}$ denotes the 
complement of $\Omega_n(z)$.
\begin{lem}\label{lem5.1}Assuming the conditions of Theorem \ref{thm1}, for $r$ such that 

$$
\ln(1/r)\,(\Delta_n^{\frac14}(z))\to\infty\quad\text{as}\quad n\to0
$$ 

we have
\begin{equation}\label{0*}
\widehat U_{\mu_n}^{(r)}\to 0,\text{  as  }n\to\infty.
\end{equation}
\end{lem}
\begin{proof}
By definition, we have
\begin{align}\label{1*}
\widehat U_{\mu_n}^{(r)}=-\frac1n\sum_{j=1}^n\E\log|\lambda_j^{(m)}-r\xi-z|I_{\Omega_n^{(c)}(z)}.
\end{align}
Applying Cauchy's inequality, we get, for any $\tau>0$,
\begin{align}\label{2*}
|\widehat U_{\mu_n}^{(r)}|&\le \frac1n\sum_{j=1}^n\E^{\frac1{1+\tau}}|\log|\lambda_j^{(m)}-r\xi-z||^{1+\tau}
\left(\Pr\{\Omega_n^{(c)}\}\right)^{\frac{\tau}{1+\tau}}\notag\\&\le
\left(\frac1n\sum_{j=1}^n\E|\log|\lambda_j^{(m)}-r\xi-z||^{1+\tau}\right)^{\frac1{1+\tau}}\left(
\Pr\{\Omega_n^{(c)}\}\right)^{\frac{\tau}{1+\tau}}.
\end{align}
Furthermore, since $\xi$ is uniformly distributed in the unit disc and independent of $\lambda_j$, we may write
\begin{equation}
\E\Big|\log|\lambda_j-r\xi-z|\,\Big|^{1+\tau}=\frac1{2\pi}\E\int_{|\zeta|\le1}
\Big|\log|\lambda_j^{(m)}-r\zeta-z|\,\Big|^{1+\tau}d\zeta=
\E J_1^{(j)}+\E J_2^{(j)}+\E J_3^{(j)},\notag
\end{equation}
where
\begin{align}
J_1^{(j)}&=\frac1{2\pi}\int_{|\zeta|\le 1,\ |\lambda_j^{(m)}-r\zeta-z|\le\varepsilon}|\log|
\lambda_j^{(m)}-r\zeta-z||^{1+\tau}d\zeta,\notag\\
J_2^{(j)}&=\frac1{2\pi}\int_{|\zeta|\le 1,\ \frac1{\varepsilon}>|\lambda_j^{(m)}-r\zeta-z|>
\varepsilon}|\log|\lambda_j^{(m)}-r\zeta-z||^{1+\tau}d\zeta,\notag\\
J_3^{(j)}&=\frac1{2\pi}\int_{|\zeta|\le 1,\ |\lambda_j^{(m)}-r\zeta-z|>\frac1{\varepsilon}}|
\log|\lambda_j^{(m)}-r\zeta-z||^{1+\tau}d\zeta.\notag
\end{align}
Note that
\begin{equation}
|J_2^{(j)}|\le\log\left(\frac1{\varepsilon}\right).\notag
\end{equation}
Since for any $b>0$, the function $-u^b\log u$ is not decreasing on the interval $[0,\exp\{-\frac1{b}\}]$,
we have for $0<u\le\varepsilon<\exp\{-\frac1{b}\}$,
\begin{equation}
-\log u\le \varepsilon^{b}u^{-b}\log\left(\frac1{\varepsilon}\right).\notag
\end{equation}
Using this inequality, we obtain, for $b(1+\tau)<2$,
\begin{align}\label{finish1}
|J_1^{(j)}|&\le\frac1{2\pi}\varepsilon^{b(1+\tau)}\left(\log\left(\frac1{\varepsilon}\right)\right)^{1+\tau}
\int_{|\zeta|\le 1,\ |\lambda_j^{(m)}-r\zeta-z|\le\varepsilon}|\lambda_j^{(m)}-r\zeta-z|^{-b(1+\tau)}d\zeta\\
&\le\frac1{2\pi r^2}\varepsilon^{b(1+\tau)}r^{-2}\log\left(\frac1{\varepsilon}\right)\int_{|\zeta|\le \varepsilon}
|\zeta|^{-b(1+\tau)}d\zeta\le C(\tau, b)\varepsilon^{2}r^{-2}\left(\log\left(\frac1{\varepsilon}\right)\right)^{1+\tau}.
\end{align}
If we choose $\varepsilon=r$, then we get
\begin{equation}\label{finish2}
|J_1^{(j)}|\le C(\tau, b)\left(\log\left(\frac1{r}\right)\right)^{1+\tau}.
\end{equation}
The following bound holds for $\frac1n\sum_{j=1}^n\E J_3^{(j)}$. Note that
$|\log x|^{1+\tau}\le \varepsilon^2|\log\varepsilon|^{1+\tau}x^2$ for $x\ge\frac1{\varepsilon}$ and 
sufficiently small $\varepsilon$.
Using this inequality, we obtain
\begin{align}\label{finish3}
\frac1n\sum_{j=1}^n\E J_3^{(j)}\le C(\tau)\varepsilon^2|\log\varepsilon|^{1+\tau}|\frac1n\sum_{j=1}^n 
\E |\lambda_j^{(\varepsilon)}-r\zeta-z|^2
\le C(\tau)(1+|z|^2+r^2)\varepsilon^2|\log\varepsilon|^{1+\tau}\notag\\
\le C(\tau)(2+|z|^2)r^2|\log r|^{1+\tau}|.
\end{align}

The inequalities (\ref{finish1})--(\ref{finish3}) together imply that
\begin{equation}\label{3*}
|\frac1n\sum_{j=1}^n\E|\log|\lambda_j^{(m)}-r\xi-z||^{1+\tau}|\le
C\left(\log\left(\frac1{r}\right)\right)^{1+\tau}.
\end{equation}
Furthermore, the inequalities (\ref{truncation}), (\ref{1*}), (\ref{2*}), and (\ref{3*}) together imply
\begin{equation}
|\widehat U_{\mu_n}^{(r)}|\le C\log\left(\frac1{r}\right)((\Delta_n^{\frac12}(z))
^{\frac{\tau}{1+\tau}}.\notag
\end{equation}
We choose $\tau=1$ and rewrite the last inequality as follows
\begin{equation}
|\widehat U_{\mu_n}^{(r)}|\le C\log\left(\frac1{r}\right)\Delta_n^{\frac14}(z)
\end{equation}
If we choose $r=\Delta_n(z)$ we obtain  
$\log(1/r)\Delta_n^{\frac14}(z)\to 0$, then (\ref{0*}) holds and the 
Lemma is proved.
\end{proof}
We shall investigate $\overline U_{\mu_n}^{(r)}$ now. Let $\nu_n(\cdots,z,r)=\E_{\zeta}\nu_n(\cdot,z+r\zeta)$ 
and $\nu(\cdot,z,r)=\E\nu(\cdot,z+r\zeta)$.
We may write
\begin{align}
\overline U_{\mu_n}^{(r)}&=-\frac1n\sum_{j=1}^n\E\log|\lambda_j^{(\varepsilon)}-z-r\xi|I_{\Omega_n(z)}=
-\frac1n\sum_{j=1}^n\E\log(s_j(\bold X^{(\varepsilon)}(z,r))I_{\Omega_n(z)}\notag\\
&=-\int_{n^{-B}}^{K_n+|z|}\log xd\E\overline F_n(x,z,r),
\end{align}
where $\overline F_n(\cdot,z,r)$ ($F(x,z,r)$ ) is the distribution function corresponding to the 
restriction of the measure $\nu_n(\cdot,z,r)$ ($\nu(\cdot,z,r)$) to the set $\Omega_n(z)$.
Introduce the notation
\begin{equation}
\overline U_{\mu}=-\int_{\Delta_n(z)}^{n+|z|}\log x dF(x,z,r).
\end{equation}
Integrating by parts, we get
\begin{align}
\overline U_{\mu_n}^{(r)}-\overline U_{\mu}&=-\int_{\Delta_n(z)}^{n+|z|}
\frac{\E F_n(x,z,r)-F(x,z,r)}x dx\notag\\&+
C\sup_x|\E F_n(x,z,r)-F(x,z,r)||\log (\Delta_n(z))|+\E\left\{\frac1n\sum_{j=k_2}^n\ln s_j I\{\Omega_n(z)\}\right\}.
\end{align}
This implies that
\begin{equation}\label{009}
|\overline U_{\mu_n}^{(r)}-\overline U_{\mu}|\le C|\log (\Delta_n(z))|\sup_x|\E F_n(x,z,r)-F(x,z)|.
\end{equation}
Note that, for any $r>0$, $|s_j(z)-s_j(z,r)|\le r$. This implies that
\begin{equation}
\E F_n(x-r,z)\le \E F_n(x,z,r)\le \E F_n(x+r,z).
\end{equation}
Hence, we get
\begin{equation}\label{supremum}
\sup_x|\E F_n(x,z,r)-F(x,z)|\le \sup_x|\E\mathcal F_n(x,z)-F(x,z)|+\sup_x|F(x+r,z)-F(x,z)|.
\end{equation}
Since the distribution function $F(x,z)$ has a density $p(x,z)$ which is bounded for $|z|>0$ and
$p(x,0)=O(x^{-\frac{m-1}{m+1}})$
(see Remark \ref{st1})
we obtain
\begin{equation}\label{08}
\sup_x|\E \mathcal F_n(x,z,r)-F(x,z)|\le \sup_x|\E \mathcal F_n^{(\varepsilon)}(x,z)-F(x,z)|+Cr^{\frac2{m+1}}.
\end{equation}
Choose $r=\Delta_n(z)$.
Inequalities (\ref{08}) and (\ref{supremum}) together imply
\begin{equation}\label{09}
\sup_x|\E\overline{\mathcal F}_n(x,z,r)-\overline F(x,z)|\le C\Delta_n^{\frac2{m+1}}(z).
\end{equation}
From inequalities (\ref{09}) and (\ref{009}) and  lemma \ref{vot} it follows that
\begin{equation}
|\overline U_{\mu_n}^{(r)}-\overline U_{\mu}|\le C\Delta_n^{\frac2{m+1}}(z)|\ln \Delta_n(z)|.\notag
\end{equation}
Note that
\begin{equation}
|\overline U_{\mu_n}^{(r)}-U_{\mu}|\le |\int_0^{\Delta_n(z)}\log xdF(x,z)|\le C\Delta_n^{\frac2{m+1}}(z)
|\ln(\Delta_n(z))|.\notag
\end{equation}

Let $\mathcal K=\{z\in\mathbb C :\ |z|\le R\}$ and let $\mathcal K^{c}$ denote $\mathbb C\setminus \mathcal K$.
According to Lemma \ref{compact}, we have, for $k_1$ and $R$ from Lemma \ref{compact},
\begin{equation}\label{rep1000}
1-q_n:=\E\mu_n^{(r)}(\mathcal K^{c})\le\frac{k_1}{n}+\Pr\{|\lambda_{k_1}|>R\}\le C\delta_n^{\frac12}(z).
\end{equation}
Furthermore, let ${\overline{\mu}}_n^{(r)}$ and ${\widehat{\mu}}_n^{(r)}$ be probability measures supported on 
the compact set
$K$ and $K^{(c)}$ respectively, such that
\begin{equation}\label{rep2}
\E\mu_n^{(r)}=q_n{\overline{\mu}}_n^{(r)}+(1-q_n){\widehat{\mu}}_n^{(r)}.
\end{equation}
 Introduce the logarithmic potential of the measure ${\overline{\mu}}_n^{(r)}$,
\begin{equation}
U_{{\overline{\mu}}_n^{(r)}}=-\int\log|z-\zeta|d{{\overline{\mu}}_n^{(r)}(\zeta)}.\notag
\end{equation}
Similar to the proof of Lemma \ref{lem5.1} we show that
\begin{equation}
|U_{\mu_n}^{(r)}-U_{{\overline{\mu}}_n^{(r)}}|\le C\Delta_n^{\frac14}(z)|\ln \Delta_n(z)|.\notag
\end{equation}
This implies that
\begin{equation}
\lim_{n\to\infty}U_{{\overline{\mu}}_n^{(r)}}(z)=U_{\mu}(z)\notag
\end{equation}
for all $z\in\mathbb C$. According to equality (\ref{main}), $U_{\mu}(z)$ is
equal to the potential of the $m$-th power of the  uniform distribution on the unit disc. This implies
that the measure $\mu$ coincides with the $m$-th power of uniform distribution on the unit disc. 
Since the measures ${\overline{\mu}}_n^{(r)}$ are compactly supported, Theorem 6.9 from \cite{saff}
and Corollary 2.2 from \cite{saff}
 together  imply that
\begin{equation}\label{rep3}
\lim_{n\to \infty}\overline{\mu}_n^{(r)}=\mu
\end{equation}
in the weak topology.
 Inequality (\ref{rep1000}) and relations (\ref{rep2}) and (\ref{rep2}) together imply
that
\begin{equation}
 \lim_{n\to\infty}\E\mu_n^{(r)}=\mu\notag
\end{equation}
in the weak topology.
Finally, by Lemma 1.1 in \cite{GT:2010}, we get
\begin{equation}
 \lim_{n\to\infty}\E\mu_n=\mu
\end{equation}
in the weak topology.
Thus Theorem \ref{main} is proved.

\section{Appendix}
Define  $\mathbf V_{\alpha,\beta}:=\prod_{\nu=\alpha}^{\beta}\bold X^{(\nu)}$.
\begin{lem}\label{mean}
 Under the conditions of Theorem \ref{main} we have, for any $j=1,\ldots,n$,
$k=1,\ldots,n$ and for any $1\le \alpha\le \beta\le m$,
\begin{equation}\notag
\E[\mathbf V_{\alpha,\beta}]_{jk}=0
\end{equation}
\end{lem}
\begin{proof}For $\alpha=\beta$ the claim is easy. Let $\alpha<\beta$ and $1\le j\le n$, $1\le k\le n$.
 Direct calculations show that
\begin{equation}\notag
 \E[\mathbf V_{\alpha,\beta}]_{jk}=\frac1{n^\frac{\beta-\alpha+1}2}\sum_{j_1=1}^{n}\sum_{j_2=1}^{n}\dots 
\sum_{j_{\beta-\alpha}=1}^{n}
\E X^{(\alpha)}_{j,j_1}X^{(\alpha+1)}_{j_1,j_2}\cdots X^{(\beta)}_{j_{\beta-\alpha-1},k}=0
\end{equation}
Thus the Lemma is proved.
\end{proof}
In all Lemmas below we shall assume that
\begin{equation}\label{as1}
 \E X_{jk}^{(\nu)}=0,\quad\E|X_{jk}^{(\nu)}|^2=1, \quad |X_{jk}^{(\nu)}|\le c\tau_n\sqrt n\quad\text{a. s.}
\end{equation}
with $\tau_n=o(1)$ such that $\tau_n^{-2}L_n(\tau_n)\le \tau_n^2$.

\begin{lem}\label{norm2}
 Assuming the  conditions of Theorem \ref{main} as well as  (\ref{as1}), we have,  for any $1\le \alpha\le \beta\le m$,
\begin{equation}
\E\|\mathbf V_{\alpha,\beta}\|_2^2\le Cn
\end{equation}

\end{lem}
\begin{proof}We shall consider the case $\alpha<\beta$ only. Other case is easy.
 Direct calculations show that
\begin{equation}\notag
 \E\|\mathbf V_{\alpha,\beta}\|_2^2\le\frac C{n^{\beta-\alpha+1}}\sum_{j=1}^n\sum_{j_1=1}^{n}\sum_{j_2=1}^{n}
\dots \sum_{j_{\beta-\alpha}=1}^{p_{\beta-\alpha-1}}\sum_{k=1}^{p_{\beta-\alpha}}
\E [X^{(\alpha)}_{j,j_1}X^{(\alpha+1)}_{j_1,j_2}\cdots X^{(\beta)}_{j_{\beta-\alpha},k}]^2
\end{equation}
By independence  of random variables, we get
\begin{equation}\notag
\E\|\mathbf V_{\alpha,\beta}\|_2^2\le  Cn
\end{equation}
Thus the Lemma is proved.
\end{proof}

\begin{lem}\label{norm4}
 Assuming the  conditions of Theorem \ref{main} as well as (\ref{as1}) we have, for any $j=1,\ldots n$,
 $k=1,\ldots, n$ and $r\ge1$,
\begin{equation}
 \E\|\mathbf V_{a,b}\mathbf e_k\|_2^{2r}\le C_r,
\end{equation}
and
\begin{equation}
 \E\|\mathbf e_j^T\mathbf V_{a,b}\|_2^{2r}\le C_r,
\end{equation}
with some positive constant $C_r$ depending on $r$.
\end{lem}
\begin{proof}
 By definition  of the matrices $\mathbf V_{a,b}$, we may write
\begin{equation}
 \|\mathbf e_j\mathbf V_{a,b}\|_2^{2}=\frac1{n^{b-a+1}}
\sum_{l=1}^{n}\left|\sum_{j_a=1}^{n}\cdots\sum_{j_{b-1}=1}^{n}
X_{jj_a}^{(a)}\cdots X_{j_{b-1}l}^{(b)}\right|^2
\end{equation}
Using this representation, we get
\begin{equation}\label{mom10}
 \E\|\mathbf V_{a,b}\mathbf e_k\|_2^{2r}=\frac1{n^{r(b-a)}}\sum_{l_1=1}^{n}\cdots\sum_{l_r=1}^{n}\E\prod_{q=1}^r
\left(\sum_{j_a=1}^{n}\cdots\sum_{j_{b-1}=1}^{n}\sum_{\widehat j_a=1}^{n}\cdots\sum_{\widehat j_{b-1}=1}^{n}
A^{(l_q)}_{(j_a,\ldots,j_b,\widehat j_1,\ldots,\widehat j_b)}\right)
\end{equation}
where
\begin{equation}\label{per}
A^{(l_q)}_{(j_a,\ldots,j_b,\widehat j_1,\ldots,\widehat  j_b)}=X_{jj_a}^{(a)}
\overline X_{j \widehat j_a}^{(a)}X_{j_aj_{a+1}}^{(a)}\overline X_{{\widehat j}_a \widehat j_{a+1}}^{(a)}
\cdots X_{j_{b-2}j_{b-1}}^{(b)} \overline X_{\widehat j_{b-2}\widehat  j_{b-1}}^{(b-1)}X_{j_{b-1}l_q}^{(b)}
\overline X_{\widehat j_{b-1}l_q}^{(b)}.
\end{equation}
By $\overline x$ we denote the  complex conjugate of the  number $x$.
 Expanding  the product on  the r.h.s of (\ref{mom10}), we get
\begin{align}
 \E\|\mathbf V_{a,b}\mathbf e_k\|_2^{2r}={\sum}^{**}\E\prod_{q=1}^rA^{(l_q)}_{(j_a^{(q)},\ldots,j_b^{(q)},
{\widehat j}_1^{(\nu)},\ldots,{\widehat j}_b^{(q)})},
\end{align}
where ${\sum}^{**}$ is taken  over all set of indices $j_{a}^{(q)},\ldots, j_{b-1}^{(q)}, l_q$ and
${\widehat j}_{a}^{(\nu)},\ldots,{\widehat j}_{b-1}^{(q)}$ where
$j_k^{(q)},{\widehat j}_k^{(q)}=1,\ldots,p_k$, $k=a,\ldots,b-1$, $l_q=1,\ldots,p_{b}$ and  $q=1,\ldots,r$.
Note that the summands in the right hand side of (\ref{per}) is equal 0 if there is at least one term in the product \ref{per}
which appears only once. This implies that  the summands in the  right hand side of (\ref{per}) 
are  not equal zero only if the union of all sets of indices in r.h.s of (\ref{per}) consist of at least $r$
 different terms and each term appears at least twice.

Introduce the following random variables, for $\nu=a+1,\ldots, b-1$,
\begin{align}
 \zeta^{(\nu)}_{j^{(1)}_{\nu-1},\ldots,j^{(r)}_{\nu-1},j^{(1)}_{\nu},\ldots,j^{(r)}_{\nu},
{\widehat j}^{(1)}_{\nu-1},\ldots,{\widehat j}^{(r)}_{\nu-1},{\widehat j}^{(1)}_{\nu},\ldots,
{\widehat j}^{(r)}_{\nu}}&
=X^{(\nu)}_{j^{(1)}_{\nu-1},j^{(1)}_{\nu}}
\cdots X^{(\nu)}_{j^{(r)}_{\nu-1},j^{(r)}_{\nu}}
{\overline X}^{(\nu)}_{{\widehat j}^{(1)}_{\nu-1},{\widehat j}^{(1)}_{\nu}},
\cdots {\overline X}^{(\nu)}_{{\widehat j}^{(r)}_{\nu-1},{\widehat j}^{(r)}_{\nu}},
\end{align}
and
\begin{align}
\zeta^{(a)}_{j^{(1)}_{1},\ldots,j^{(r)}_{1},
{\widehat j}^{(1)}_{1},\ldots,{\widehat j}^{(r)}_{1}}&=X^{(a)}_{jj_1^{(a)}}
\cdots X^{(a)}_{j^{(r)}_{a}j^{(r)}_{a+1}}
{\overline X}^{(a)}_{j{\widehat j}^{(1)}_{a}}
\cdots {\overline X}^{(a)}_{{\widehat j}^{(r)}_{a},{\widehat j}^{(r)}_{a+1}}
\notag\\
\zeta^{(b)}_{j^{(1)}_{b-1},\ldots,j^{(r)}_{b-1},
{\widehat j}^{(1)}_{b-1},\ldots,{\widehat j}^{(r)}_{b-1},l_q}
&=X^{(b)}_{j^{(1)}_{b-1}j^{(1)}_{b}}
\cdots X^{(b)}_{j^{(r)}_{b-1}l_q}
{\overline X}^{(b)}_{{\widehat j}^{(1)}_{b-1},l_q},
\cdots {\overline X}^{(b)}_{{\widehat j}^{(r)}_{b-1},l_q}.
\notag
\end{align}
 Let the set of indices  $j^{(1)}_{a},\ldots,j^{(r)}_{a},
{\widehat j}^{(1)}_{a},\ldots,{\widehat j}^{(r)}_{a}$ contain $t_a$ different indices, say
$i_1^{(a)},\ldots,i_{t_a}^{(a)}$ with multiplicities $k_1^{(a)},\ldots,k_{t_a}^{(a)}$
respectively, $k_1^{(a)}+\ldots+k_{t_a}^{(a)}=2r$.
Note that \newline $\min\{k_1^{(a)},\ldots,k_{t_a}^{(a)}\}\le 2$. Otherwise, $|\E\zeta^{(a)}_{j^{(1)}_{a},
\ldots,j^{(r)}_{a},
{\widehat j}^{(1)}_{a},\ldots,{\widehat j}^{(r)}_{a}}|=0$. By assumption (\ref{as1}), we have
\begin{equation}\label{mom1}
|\E\zeta^{(a)}_{j^{(1)}_{a},\ldots,j^{(r)}_{a},
{\widehat j}^{(1)}_{a},\ldots,{\widehat j}^{(r)}_{a}}|\le C(\tau_n\sqrt n)^{2r-2t_a}
\end{equation}
A similar bound we get for $|\E
\zeta^{(b)}_{j^{(1)}_{b-1},\ldots,j^{(r)}_{1},
{\widehat j}^{(1)}_{b-1},\ldots,{\widehat j}^{(r)}_{b-1},l_q}|$. Assume that
the set of indices $\{j^{(1)}_{b-1},\ldots,j^{(r)}_{b-1}$,
${\widehat j}^{(1)}_{b-1},\ldots,{\widehat j}^{(r)}_{b-1}\}$ contains $t_{b-1}$ different indices, say,
$i_1^{(b-1)},\ldots,i_{t_{b-1}}^{(a)}$ with multiplicities\newline $k_1^{(b-1)},\ldots,k_{t_{b-1}}^{(a)}$
respectively, $k_1^{(b-1)}+\ldots+k_{t_{b-1}}^{(a)}=2r$. Then
\begin{equation}\label{mom2}
|\E
\zeta^{(b)}_{j^{(1)}_{b-1},\ldots,j^{(r)}_{1},
{\widehat j}^{(1)}_{b-1},\ldots,{\widehat j}^{(r)}_{b-1},l_q}|\le C(\tau_n\sqrt n)^{2r-2t_{b-1}}
\end{equation}
Furthermore, assume that for $a+1\le \nu\le b-2$ there are $t_{\nu}$
different pairs of  indices, say, $(i_{a},i'_a),
\ldots(i_{t_{b}},i'_{t_{b}})$ in the set\newline $\{j^{(1)}_{a},\ldots,j^{(r)}_{a},
{\widehat j}^{(1)}_{a},\ldots,{\widehat j}^{(r)}_{a},\ldots,j^{(1)}_{b-1},\ldots,j^{(r)}_{b-1},
{\widehat j}^{(1)}_{b-1},\ldots,{\widehat j}^{(r)}_{b-},l_1,l_r\}$
with multiplicities\newline$k_1^{(\nu)},\ldots,k_{t_{\nu}}^{(\nu)}$.
Note that
\begin{equation}
 k_1^{(\nu)}+\ldots+k_{t_{\nu}}^{(\nu)}=2r
\end{equation}
and
\begin{equation}\label{mom3}
\E\zeta^{(\nu)}_{j^{(1)}_{\nu-1},\ldots,j^{(r)}_{\nu-1},j^{(1)}_{\nu},\ldots,j^{(r)}_{\nu},
{\widehat j}^{(1)}_{\nu-1},\ldots,{\widehat j}^{(r)}_{\nu-1},{\widehat j}^{(1)}_{\nu},\ldots,
{\widehat j}^{(r)}_{\nu}}\le C(\tau_n\sqrt n)^{2r-2t_{\nu}}.
\end{equation}
The inequalities (\ref{mom1})-(\ref{mom3}) together yield
\begin{equation}\label{mom11}
 |\E\prod_{q=1}^rA^{(l_q)}_{(j_a^{(q)},\ldots,j_b^{(q)},{\widehat j}_1^{(\nu)},\ldots,{\widehat j}_b^{(q)})}|
\le C(\tau_n\sqrt n)^{2r(b-a)-2(t_1+\ldots+t_{b-a})}.
\end{equation}
It is straightforward to check that the number $\mathcal N(t_a,\ldots,t_{b})$ of sequences of indices
\newline$\{j^{(1)}_{a},\ldots,j^{(r)}_{a},
{\widehat j}^{(1)}_{a},\ldots,{\widehat j}^{(r)}_{a},\ldots,j^{(1)}_{b-1},\ldots,j^{(r)}_{b-1},
{\widehat j}^{(1)}_{b-1},\ldots,{\widehat j}^{(r)}_{b-},l_1,\ldots,l_r\}$ with $t_a,\ldots,t_{b}$
of different pairs  satisfies the inequality
\begin{equation}\label{finish}
 \mathcal N(t_a,\ldots,t_{b})\le Cn^{t_a+\ldots+t_b},
\end{equation}
with $1\le t_i\le r,\quad i=a,\ldots,b$.
Note that in the case $t_a=\cdots=t_b=r$ the inequalities (\ref{mom1})--(\ref{mom3}) imply
\begin{equation}\label{mom4}
 \E\zeta^{(\nu)}_{j^{(1)}_{\nu-1},\ldots,j^{(r)}_{\nu-1},j^{(1)}_{\nu},\ldots,j^{(r)}_{\nu},
{\widehat j}^{(1)}_{\nu-1},\ldots,{\widehat j}^{(r)}_{\nu-1},{\widehat j}^{(1)}_{\nu},\ldots,
{\widehat j}^{(r)}_{\nu}}\le C
\end{equation}
The inequalities  (\ref{finish}),  (\ref{mom11}), (\ref{mom4}), and representation (\ref{mom10}) 
together conclude the proof.
\end{proof}

{\bf The Largest Singular Value.}
Recall that $|\lambda_1^{(m)}|\ge\ldots\ge|\lambda_n^{(m)}|$ denotes the eigenvalues of 
the matrix $\bold W$ ordered by  decreasing absolute values and let $s_1^{(m)}\ge\ldots\ge
s_n^{(m)}$ denote the singular values of the matrix  $\bold W$.

We show the following 
\begin{lem}\label{largeval}Under the  conditions of Theorem \ref{main} we have, 
 for sufficiently large $K\ge 1$ 
\begin{equation}
\Pr\{s_1^{(m)}\ge n\}\le C/{n}
\end{equation}
for some positive constant $C>0$.
\end{lem}
\begin{proof}Using Chebyshev's inequality, we get
\begin{equation}
 \Pr\{s_1^{(m)}\ge n\}\le \frac1{n^2}
\E\Tr\Big(\bold W\bold W^*\Big)\le
\frac1{n}
\end{equation}
Thus the Lemma is  proved. 
\end{proof}

\begin{lem}\label{var0}Under conditions of Theorem \ref{main} assuming (\ref{as1}), we have
\begin{equation}\notag
\E|\frac1n(\Tr \mathbf R-\E\Tr \mathbf R)|\le \frac C{nv^2}.
\end{equation}
\end{lem}
\begin{proof} 
Consider the matrix $\mathbf X^{(1,j)}$ obtained from the matrix $\mathbf
X^{(1)}$  by replacing its  $j$-th row by a row with zero-entries. We define the following matrices
\begin{equation} \notag\mathbf H^{(\nu,j)}=\mathbf H^{(\nu)}-\mathbf e_j\mathbf
e_j^T\mathbf H^{(\nu)},
\end{equation}
and
\begin{equation}\notag
{\widetilde{\mathbf H}}^{(m-\nu+1,j)}={{\mathbf
H}}^{(m-\nu+1)}-{{\mathbf H}}^{(m-\nu+1)}\mathbf
e_{j+n}\mathbf e_{j+n}^T.
\end{equation}
For the simplicity  we shall assume that  $\nu\le m-\nu+1$.
 Define $$\mathbf
V^{(\nu,j)}=\prod_{q=1}^{\nu-1}\mathbf H^{(q)}\,\mathbf
H^{(\nu,j)}\prod_{q=\nu+1}^{m-\nu}\mathbf H^{(q)}{\widetilde{\mathbf
H}}^{(m-\nu+1,j)}\prod_{q=m-\nu+2}^{m}\mathbf H^{(q)}.$$
Let $\mathbf V{(\nu,j)}(z)=\mathbf V{(\nu,j)}\mathbf J-\mathbf J(z)$.  We
shall use the following inequality. For any Hermitian matrices  $\mathbf A$ and $\mathbf
B$ with spectral distribution function $F_A(x)$ and $F_B(x)$
respectively, we have
\begin{equation}\label{trace}
|\Tr (\mathbf A-\alpha\mathbf I)^{-1}-\Tr (\mathbf B-\alpha\mathbf I)^{-1}|\le
\frac {\text{\rm rank}(\mathbf A-\mathbf B)}{v},
\end{equation}
where $\alpha=u+iv$.
It is straightforward to show that
\begin{equation}\label{rank}\text{\rm rank}(\mathbf V(z)-\mathbf V^{(\nu,j)}(z))=
\text{\rm rank}(\mathbf V\mathbf J-\mathbf V^{(\nu,j)}\mathbf J)\le 4m.
\end{equation}
Inequality (\ref{trace}) and (\ref{rank}) together imply
\begin{equation}\notag
|\frac1{2n}(\Tr \mathbf R-\Tr \mathbf R^{(\nu,j)})|\le \frac C{nv}.
\end{equation}

After this remark we may apply a standard martingale expansion
procedure. We  introduce
$\sigma$-algebras $\mathcal F_{\nu,j}=\sigma\{X^{(\nu)}_{lk},\,
j< l\le n, k=1,\ldots,n; X^{(q)}_{pk}$,
$q=\nu+1,\ldots m, \,p=1,\ldots,n,\, k=1,\ldots,n\}$ and use the representation
\begin{equation}\notag
\Tr\mathbf R-\E\Tr\mathbf R=\sum_{\nu=1}^m\sum_{j=1}^{n}(\E_{\nu,j-1}\Tr\mathbf R-\E_{\nu,j}\Tr\mathbf R),
\end{equation}
where $\E_{\nu,j}$ denotes  conditional expectation given the  $\sigma$-algebra $\mathcal F_{\nu,j}$.
 Note that $\mathcal F_{\nu,n}=\mathcal F_{\nu+1,0}$
\end{proof}
\begin{lem}\label{var1}
Under the conditions of Theorem \ref{main} we have, for $1\le a,\le m$,
\begin{equation}\notag
\E|\frac1n(\sum_{k=1}^{n}[\mathbf V_{a+1,m}\mathbf J\mathbf R\mathbf
V_{1,m-a}]_{k,k+n}-\E\sum_{j=1}^{n}[\mathbf V_{a+1,m}\mathbf J\mathbf R\mathbf V_{1,m-a}]_{kk+{n}})|^2\le
 \frac C{n v^4}.
\end{equation}
and, for $1\le a,\le m-1$,
\begin{equation}\notag
\E|\frac1n(\sum_{k=1}^{n}[\mathbf V_{m-a+2,m}\mathbf J\mathbf R\mathbf
V_{1,m-a+1}]_{k,k}-\E\sum_{j=1}^{n}[\mathbf V_{m-a+2,m}\mathbf J\mathbf R\mathbf V_{1,m-a+1}]_{kk})|^2
\le \frac C{n v^4}.
\end{equation}
\end{lem}
\begin{proof}We prove the first inequality only. The proof of the other one is  similar.
 For $\nu=1,\ldots,m$ and for
$j=1,\ldots,n$, we introduce the  matrices,
 $\mathbf X^{(\nu,j)}=\mathbf X^{(\nu)}-\mathbf
e_j\mathbf e_j^T\mathbf X^{(\nu)}$,
 and $\mathbf H^{(\nu,j)}=\mathbf
H^{(\nu)}-\mathbf e_j\mathbf e_j^T\mathbf H^{(\nu)}$ and \newline
${\widetilde{\mathbf H}}^{(m-\nu+1,j)}=\mathbf H^{(m-\nu+1,j)}-\mathbf
H^{(m-\nu+1)}\mathbf e_{j+n} \mathbf e_{j+n}^T$.
 Note that the matrix $\mathbf X^{(\nu,j)}$ is
obtained from the matrix $\mathbf X^{(\nu)}$ by replacing its
$j$-th row by a row of zeros. Similar to the proof  of the 
previous Lemma we introduce
the matrices $\mathbf V^{(\nu,j)}_{c,d}$ by replacing in the definition of  
$\mathbf V_{c,d}$ the  matrix $\mathbf H^{(\nu)}$ by $\mathbf
H^{(\nu,j)}$ and the matrix $\mathbf H^{(m-\nu+1)}$ by ${\widetilde{\mathbf
H}}^{(m-\nu+1,j)}$. For instance, if $c\le\nu\le m-\nu+1\le d$ we get
\begin{equation}\notag
\mathbf V^{(\nu,j)}_{c,d}=\prod_{q=a}^{\nu-1}\mathbf H^{(q)}\,\mathbf
H^{(\nu,j)}\prod_{q=\nu+1}^{m-\nu}\mathbf H^{(q)}{\widetilde{\mathbf
H}}^{(m-\nu+1,j)}\prod_{q=m-\nu+1}^{b}\mathbf H^{(q)}
\end{equation}.

 Let $\mathbf
V^{(\nu,j)}:= \mathbf V_{1,m}^{(\nu,j)}$ and
$\mathbf R^{(j)}:= (\mathbf V^{(\nu,j)}(z)-\alpha\mathbf I)^{-1}$. Introduce  the
following quantities, for $\nu=1\ldots,m$ and $j=1,\ldots,n$,
\begin{equation}\notag
\Xi_j:=\sum_{k=1}^n[\mathbf V_{a+1,m}\mathbf J\mathbf R\mathbf
V_{1,m-a+1}]_{kk+n}- \sum_{k=1}^n[\mathbf V^{(\nu,j)}_{a+1,m}\mathbf J\mathbf
R^{(\nu,j)}\mathbf V^{(\nu,j)}_{1,m-a+1}]_{kk+n}
\end{equation}
We represent them  in the following form
\begin{equation}\notag
\Xi_j:= \Xi_j^{(1)}+\Xi_j^{(2)}+\Xi_j^{(3)},
\end{equation}
where
\begin{align}
\Xi_{\nu,j}^{(1)}&= =\sum_{k=1}^{n}[(\mathbf V_{a+1,m}
-\mathbf V_{a+1,m}^{(\nu,j)})\mathbf J\mathbf R\mathbf  V_{1,m-a+1}]_{k,k+n},\notag\\
\Xi_{\nu,j}^{(2)}&= \sum_{k=1}^{n}[\mathbf V_{a+1,m}^{(\nu,j)}\mathbf J(\mathbf R-\mathbf R^{(\nu,j)})
\mathbf J\mathbf  V_{1,m-a+1}]_{kk+n},\notag\\
\Xi_{\nu,j}^{(3)}&= \sum_{k=1}^{n}[\mathbf
V^{(j)}_{a+1,m}\mathbf J\mathbf R^{(\nu,j)}(\mathbf  V_{1,m-a+1}-\mathbf
V_{1,m-a+1}^{(\nu,j)})]_{kk+n}.\notag
\end{align}
Note that
\begin{align}
 \mathbf V_{a+1,m}-\mathbf V^{(\nu,j)}_{a+1,m}&=\mathbf V_{a+1,\nu-1}(\mathbf H^{(\nu)}-\mathbf
 H^{(\nu,j)})\mathbf V_{\nu+1,m}\notag\\&+
 \mathbf V_{a+1,\nu-1}\mathbf H^{(\nu,j)}\mathbf
 V_{\nu+1,m-\nu}(\widetilde {\mathbf H}_{m-\nu+1}-{\widetilde {\mathbf
 H}}_{m-\nu+1}^{\nu,j})\mathbf V_{m-\nu+2,m}.\notag
\end{align}
By definition of the matrices $\mathbf H^{\nu,j}$ and ${\widetilde{\mathbf
H}}^{m-\nu+1,j}$, we have
\begin{align}
\sum_{k=1}^{n}[(\mathbf V_{a+1,m} -\mathbf
V_{a+1,m}^{(\nu,j)})\mathbf J\mathbf R\mathbf
V_{1,m-\nu+1}]_{k,k+n}=[\mathbf V_{\nu+1,m}\mathbf J\mathbf R\mathbf V_{1,m-a+1}\mathbf{\widetilde J}
\mathbf V_{a+1,\nu}]_{j,j}&\notag\\
+[\mathbf V_{m-\nu+2,m}\mathbf J\mathbf R\mathbf V_{1,m-a+1}\mathbf{\widetilde
J}\mathbf V_{a+1,m-a+1}]_{j+n,j+n},&\notag
\end{align}
where
$$
\mathbf{\widetilde J}=\left(\begin{matrix}{\mathbf O\quad\mathbf
I}\\{\mathbf O\quad\mathbf O}\end{matrix}\right)
$$

This equality implies that
\begin{align}
 |\Xi_j^{(1)}|&\le |[\mathbf V_{\nu+1,m}\mathbf J\mathbf R\mathbf V_{1,m-a+1}\mathbf{\widetilde J}
\mathbf V_{a+1,\nu}]_{j,j+n}|\notag\\&\qquad\qquad+|[\mathbf V_{m-\nu+2,m}\mathbf J\mathbf R
\mathbf V_{1,m-a+1}\mathbf{\widetilde
J}\mathbf V_{a+1,m-\nu+1}]_{j+n,j+n}|.\notag
\end{align}

Using the obvious  inequality $\sum_{j=1}^n a_{jj}^2\le \|\mathbf A\|_2^2$ for any matrix 
$\mathbf A=(a_{jk})$,
$j,k=1,\ldots,n$, we get
\begin{align}
T_1:=\sum_{j=1}^n\E|\Xi_j^{(1)}|^2\le &\E\|\mathbf V_{\nu+1,m}\mathbf J\mathbf R\mathbf V_{1,m-a+1}
\mathbf{\widetilde J}\mathbf V_{a+1,\nu}\|_2^2\notag\\&+\E\|\mathbf V_{m-\nu+2,m}\mathbf J\mathbf R
\mathbf V_{1,m-a+1}\mathbf{\widetilde
J}\mathbf V_{a+1,m-\nu+1}\|_2^2.\notag
\end{align}

By Lemma \ref{norm2}, we get
\begin{equation}\label{T1}
 T_1\le \frac{C}{v^2}\E\|\mathbf V_{a+1,m}\mathbf V_{1,m-a+1}\|_2^2\le \frac {Cn}{v^2}
\end{equation}

Consider now the term
\begin{equation}\notag
 T_2=\sum_{j=1}^n\E|\Xi_j^{(2)}|^2.
\end{equation}
Using that $\mathbf R-\mathbf R^{(j)}=-\mathbf R^{(j)}(\mathbf V(z)-\mathbf
V^{(\nu,j)}(z))\mathbf R$, we get
\begin{align}
 |\Xi_{j}^{(2)}|&\le |\sum_{k=1}^{n}[\mathbf V^{(\nu,j)}_{a,m}\mathbf J\mathbf R\mathbf V_{1,\nu-1}
\mathbf e_j\mathbf e_j^T
 \mathbf V_{\nu,m}
\mathbf R\mathbf V_{1,b}]_{k,k+n}|\notag\\&\qquad\le [\mathbf
J\mathbf H^{(\alpha+1)}\mathbf V_{\alpha+2,m-\alpha}\mathbf
H^{(m-\alpha+1,j)} \mathbf V_{m-\alpha+2,m}\mathbf R\mathbf
V_{1,m-\alpha}\mathbf V^{(j)}_{\alpha+1,m}\mathbf J\mathbf R\mathbf
V_{1,\alpha}]_{jj} .\notag
\end{align}
This implies that
\begin{equation}\notag
 T^{(2)}\le C\E\|[\mathbf V_{\nu+1,m}\mathbf J\mathbf R\mathbf V_{1,b}\mathbf V_{a,m}\mathbf J\mathbf R\mathbf
\mathbf V_{1,\nu}\|_2^2.
\end{equation}
It is straightforward to check that
\begin{equation}\label{t2}
 T^{(2)}\le \frac C{v^4}\E\|\mathbf V_{1,\alpha}\mathbf J\mathbf H^{(\alpha+1)}
\mathbf V_{\alpha+2,m-\alpha}\mathbf H^{(m-\alpha+1,j)}
\mathbf V_{m-\alpha+2,m}\|_2^2=\E\|\mathbf Q\|_2^2
\end{equation}
The matrix on the right hand side of equation (\ref{t2}) may be  represented
 in the following form
\begin{equation}\notag
 Q=\prod_{\nu=1}^m{\mathbf H^{(\nu)}}^{\varkappa_{\nu}},
\end{equation}
where $\varkappa_{\nu}=0$ or $\varkappa_{\nu}=1$ or $\varkappa_{\nu}=2$.
Since $X^{(\nu)}_{ss}=0$, for $\varkappa=1$ or $\varkappa=2$, we have
\begin{equation}\notag
 \E|{\mathbf H^{(\nu)}}^{\varkappa}_{kl}|^2\le \frac C{n}.
\end{equation}
This implies that
\begin{equation}\label{T2}
 T_2\le Cn.
\end{equation}
Similar we prove that
\begin{equation}\label{T3}
 T_3:=\sum_{j=1}^n\E|\Xi_j^{(3)}|^2\le Cn.
\end{equation}
Inequalities (\ref{T1}), (\ref{T2}) and (\ref{T3}) together imply
\begin{equation}\notag
 \sum_{j=1}^n\E|\Xi_j|^2\le Cn
\end{equation}
Applying now a martingale expansion with respect to the $\sigma$-algebras
$\mathcal F_j$ generated by the random variables
$X_{kl}^{(\alpha+1)}$ with $1\le k\le j$, $1\le l\le n$ and all other random variables $X^{(q)}_{sl}$
except $q=\alpha+1$, we get
\begin{equation}\notag
\E|\frac1n(\sum_{k =1}^n[\mathbf V_{\alpha+1,m}\mathbf J\mathbf R
\mathbf  V_{1,m-\alpha}]_{kk+n}-\E\sum_{j=1}^n[\mathbf V_{\alpha+1,m}\mathbf J\mathbf R
\mathbf V_{1,m-\alpha}]_{kk+n})|^2\le \frac C{n v^4}.
\end{equation}

Thus the Lemma is proved.

\end{proof}

\begin{lem}\label{derivatives}
Under the conditions of Theorem \ref{main} we have, for $\alpha=1,\ldots,m,$
there exists a constant  $C$  such that
\begin{equation}\notag
\frac1{n^{\frac32}}\E\left|\sum_{j=1}^{n}\sum_{k=1}^{n}(-X^{(\alpha)}_{jk}+(1-\theta_{jk}){X^{(\alpha)}_{jk}}^3)
\left[\frac{\partial ^{2}(\mathbf V_{\alpha+1,m}\mathbf J\mathbf R\mathbf V_{1,m-\alpha+1})}
{\partial {X_{jk}^{(\alpha)}}^{2}}(\theta_{jk}^{(\alpha)}X_{jk}^{(\alpha)})\right]_{kj}\right|
\le C\tau_nv^{-4},
\end{equation}
and
\begin{align}\notag
\frac1{n^{\frac32}}&\E\left|\sum_{j=1}^{p_{m-\alpha}}\sum_{k=1}^{p_{m-\alpha+1}}
(X^{(m-\alpha+1)}_{jk}+{X^{(m-\alpha+1)}_{jk}}^3)\right.\notag\\ &\qquad\qquad\left.\times\left[\frac{\partial ^{2}
(\mathbf V_{\alpha+1,m}\mathbf J\mathbf R\mathbf V_{1,m-\alpha+1})}{\partial {X_{jk}^{(m-\alpha+1)}}^{2}}
(\theta_{jk}^{(m-\alpha+1)}X_{jk}^{(m-\alpha+1)})\right]_{j+n,k}\right|
\le C\tau_nv^{-4},
\end{align}
where $\theta_{jk}^{(\alpha)}$ and $X_{jk}^{(\alpha)}$ are independent in aggregate for
$\alpha=1,\ldots,m$ and $j=1,\ldots,n$, $k=1,\ldots,n$, and
$\theta_{jk}^{(\alpha)}$ are r.v. which are  uniformly distributed on the unit interval.\newline
By $\frac{\partial^2}{\partial {X_{jk}^{(\alpha)}}^2}\mathbf A(\theta_{jk}^{(\alpha)}X_{jk}^{(\alpha)})$ 
we denote the matrix obtained from
$\frac{\partial^2}{\partial {X_{jk}^{(\alpha)}}^2}\mathbf A$ by replacing its entries
$X_{jk}^{(\alpha)}$ by $\theta_{jk}^{(\alpha)}X_{jk}^{(\alpha)}$.
\end{lem}
\begin{proof}The proof of this lemma is rather technical. But we  shall include it
 for completeness.
By the formula for the derivatives of a resolvent matrix, we have
\begin{equation}\label{deriv}
\frac{\partial (\mathbf V_{\alpha+1,m}\mathbf J\mathbf R\mathbf V_{1,m-\alpha+1})}{\partial X_{jk}^{(\alpha)}}=
\sum_{l=1}^5Q_l,
\end{equation}
\begin{align}
\mathbf Q_1=&\frac1{\sqrt n}
\mathbf V_{\alpha+1,m}\mathbf J\mathbf R\mathbf V_{1,\alpha-1}\mathbf e_j\mathbf e_k^T\mathbf V_{\alpha+1,m-\alpha+1}
I_{\{\alpha\le m-\alpha+1\}})\notag\\
\mathbf Q_2=&\frac1{\sqrt n}\mathbf V_{\alpha+1,m}\mathbf J\mathbf R\mathbf V_{1,m-\alpha}
\mathbf e_{k+n}\mathbf e_{j+n}\notag\\
\mathbf Q_3=&-\frac1{\sqrt n}\mathbf V_{\alpha+1,m}\mathbf J\mathbf R\mathbf V_{1,\alpha-1}\mathbf e_j
\mathbf e_k^T\mathbf V_{\alpha+1,m}\mathbf J\mathbf R\mathbf V_{1,m-\alpha+1}\notag\\
\mathbf Q_4=&-\frac1{\sqrt n}\mathbf V_{\alpha+1,m}\mathbf J\mathbf R\mathbf V_{1,m-\alpha}
\mathbf e_{k+p_{m-\alpha}}\mathbf e_{j+p_{m-\alpha+1}}^T\mathbf V_{m-\alpha+2,m}\mathbf J\mathbf R
\mathbf V_{1,m-\alpha+1}\notag\\
\mathbf Q_5=&\frac1{\sqrt n}\mathbf V_{\alpha+1,m-\alpha}\mathbf e_{k+p_{m-\alpha}}\mathbf e_{j+n}^T
\mathbf V_{m-\alpha+2,m}\mathbf J\mathbf R\mathbf V_{1,m-\alpha+1}I_{\{\alpha\le m-\alpha+1\}}).\notag
\end{align}
Introduce the notations
\begin{equation}\notag
 \mathbf U_{\alpha}:=\mathbf V_{\alpha+1,m},\quad\mathbf V_{\alpha}=\mathbf V_{1,m-\alpha+1}.
\end{equation}
From  formula (\ref{deriv}) it follows that
\begin{equation}\notag
 \frac{\partial^2 (\mathbf U_{\alpha}\mathbf J\mathbf R\mathbf V_{\alpha})}
{\partial {X_{jk}^{(\nu)}}^2}=\sum_{l=1}^{5}\frac{\partial \mathbf Q_l}{\partial X_{jk}^{(\alpha)}}.
\end{equation}
Since all the calculations will be similar we consider the case  $l=3$ only.
Simple calculations of derivatives show that  
\begin{equation}
  \frac{\partial \mathbf Q_3}{\partial X_{jk}^{(\alpha)}}=\sum_{m=1}^{7} \mathbf P^{(m)},
\end{equation}
 where
\begin{align}
\mathbf P^{(1)}&=-\frac1n\mathbf V_{\alpha+1,m-\alpha}\mathbf e_{k+p_{m-\alpha}}\mathbf e_{j+n}^T
\mathbf U_{m-\alpha+1}\mathbf J\mathbf R\mathbf V_{m-\alpha+2}\mathbf e_{j}\mathbf e_{k}^T
\mathbf U_{\alpha}\mathbf J\mathbf R\mathbf V_{\alpha}\notag\\
\mathbf P^{(2)}&=-\frac1n\mathbf U_{\alpha}\mathbf J\mathbf R\mathbf V_{m-\alpha+2}\mathbf e_{j}
\mathbf e_{k}^T\mathbf U_{\alpha}\mathbf J\mathbf R\mathbf V_{\alpha+1}\mathbf e_{k+n}\mathbf e_{j+n}^T
\notag\\
\mathbf P^{(3)}&=-\frac1n\mathbf U_{\alpha}\mathbf J\mathbf R\mathbf V_{m-\alpha+2}
\mathbf e_{j}\mathbf e_{k}^T\mathbf V_{\alpha+1,m-\alpha}\mathbf e_{k+n}\mathbf e_{j+n}^T
\mathbf U_{m-\alpha+1}\mathbf J\mathbf R\mathbf V_{\alpha}\notag\\
\mathbf P^{(4)}&=-\frac1n\mathbf U_{\alpha}\mathbf J\mathbf R\mathbf V_{m-\alpha+2}
\mathbf e_j\mathbf e_k^T\mathbf U_{\alpha}\mathbf J\mathbf R\mathbf V_{m-\alpha+2}\mathbf e_j
\mathbf e_k^T\mathbf U_{\alpha}\mathbf J\mathbf R\mathbf V_{\alpha}\notag\\
\mathbf P^{(5)}&=\frac1n\mathbf U_{\alpha}\mathbf J\mathbf R\mathbf V_{\alpha+1}\mathbf e_{k+n}
\mathbf e_{j+n}^T\mathbf U_{m-\alpha+1}\mathbf J\mathbf R\mathbf V_{m-\alpha+2}\mathbf e_j\mathbf e_k^T
\mathbf U_{\alpha}\mathbf J\mathbf R\mathbf V_{\alpha}
\notag\\
\mathbf P^{(6)}&=\frac1n\mathbf U_{\alpha}\mathbf J\mathbf R\mathbf V_{m-\alpha+2}\mathbf e_j\mathbf e_k^T
\mathbf U_{\alpha}\mathbf J\mathbf R\mathbf V_{\alpha+1}\mathbf e_{k+n}\mathbf e_{j+n}^T\mathbf U_{m-\alpha+1}
\mathbf J\mathbf R\mathbf V_{\alpha}
\notag\\
\mathbf P^{(7)}&=\frac1n\mathbf U_{\alpha}\mathbf J\mathbf R\mathbf V_{m-\alpha+2}\mathbf e_j\mathbf e_k^T
\mathbf U_{\alpha}\mathbf J\mathbf R\mathbf V_{m-\alpha+2}\mathbf e_{j}\mathbf e_{k}^T\mathbf U_{\alpha}
\mathbf J\mathbf R\mathbf V_{\alpha}.\notag
\end{align}

Consider now the quantity, for $\mu=1,\ldots,5$,
\begin{equation}\label{fin1}
 L_{\mu}=\frac1{n^{\frac32}}\sum_{j=1}^{n}\sum_{k=1}^{n}\E {X_{j,k}^{(\alpha)}}^3
\left[\frac{\partial \mathbf Q_{\mu}}{\partial X_{jk}^{(\alpha)}}\right]_{kj}.
\end{equation}
We bound $L_3$ only. The others terms are bounded in a similar way.
First we note that
\begin{equation}\label{fin2}
 \sum_{j=1}^{n}\sum_{k=1}^{n}\E {X_{j,k}^{(\alpha)}}^3[\mathbf P^{(\nu)}]_{kj}=0,\quad\text{for}\quad \nu=1,2,3.
\end{equation}
Furthermore,
\begin{equation}
 \E |{X_{j,k}^{(\alpha)}}^3||[\mathbf P^{(4)}]_{kj}|\le\E|X_{jk}^{(\alpha)}|^3
|[\mathbf U_{\alpha}\mathbf J\mathbf R\mathbf V_{m-\alpha+2}]_{kj}|^2|[\mathbf U_{\alpha}\mathbf J\mathbf R
\mathbf V_{\alpha}]_{kj}|.
\end{equation}

Let $\mathbf U_{\alpha}^{(jk)}$ ( $\mathbf V^{(j,k)}_{\alpha}$) denote matrix obtained from $\mathbf U_{\alpha}$ 
 ($\mathbf V_{\alpha}$) by replacing $X_{jk}^{(\alpha)}$ by zero. We may write
\begin{align}
\mathbf U_{\alpha}= \mathbf U_{\alpha}^{(jk)}+\frac1{\sqrt n}X_{jk}^{(\alpha)}\mathbf V_{\alpha+1,m-\alpha+1}
\mathbf e_{k+n}\mathbf e_{j+n}^T\mathbf V_{m-\alpha+2,m}.
\end{align}
and
\begin{align}
 \mathbf V_{\alpha}=\mathbf V^{(j,k)}_{\alpha}+\frac1{\sqrt n}X_{jk}\mathbf V_{1,m-\alpha+1}
\mathbf e_{k+n}\mathbf e_{j+n}^T.\notag
\end{align}
 Using these representations and taking in account that
\begin{equation}
 [\mathbf V_{\alpha+1,m-\alpha}]_{k,k+n}=[\mathbf V_{1,m-\alpha}]_{k,k+n}=0,
\end{equation}
 we get by differentiation
\begin{align}\label{fin8}
 \E |{X_{j,k}^{(\alpha)}}^3||[\mathbf P^{(4)}]_{kj}|\le \frac1{n}\E |{X_{j,k}^{(\alpha)}}^3|\,
|[\mathbf U_{\alpha}\mathbf J\mathbf R\mathbf V_{m-\alpha+2}]_{kj}|^2 \,
|[\mathbf U^{(j,k)}_{\alpha}\mathbf J\mathbf R\mathbf V^{(j,k)}_{\alpha}]_{kj}|.
\end{align}
Furthermore,
\begin{align}\label{fin9}
|[\mathbf U_{\alpha}\mathbf J\mathbf R\mathbf V_{m-\alpha+2}]_{k,j}|&\le
\frac1v\|\mathbf V_{m-\alpha+2}\mathbf e_{j}\|_2\|\mathbf e_k^T\mathbf U_{\alpha}\|_2\notag\\
|[\mathbf U^{(j,k)}_{\alpha}\mathbf J\mathbf R\mathbf V^{(j,k)}_{\alpha}]_{kj}|&\le\frac1v
\|\mathbf V^{(j,k)}_{\alpha}\mathbf e_{k}\|_2\|\mathbf e_j^T\mathbf U^{(j,k)}_{\alpha}\|_2.\notag\\
\end{align}
Applying inequalities (\ref{fin8}) and (\ref{fin9}) and taking in account the 
 independence of entries, we get
\begin{align}
\E |{X_{j,k}^{(\alpha)}}^3||[\mathbf P^{(4)}]_{kj}|\le \frac1{nv^2}\E |{X_{j,k}^{(\alpha)}}^3 
\E\|\mathbf V_{m-\alpha+2}\mathbf e_{k}\|_2^2\|\mathbf e_j^T\mathbf U_{\alpha}\|_2^2
\|\mathbf V^{(j,k)}_{\alpha}\mathbf e_{k}\|_2\|\mathbf e_j^T\mathbf U^{(j,k)}_{\alpha}\|_2
\end{align}
Applying Lemma \ref{norm4}, we get
\begin{equation}
 \frac1{n^{\frac32}}\sum_{j=1}^{n}\sum_{k=1}^{n}\E |X_{jk}^{(\alpha)}|^3
|[\mathbf P^{(4)}]_{kj}|\le \frac{C}{n^{\frac52}}\sum_{j=1}^{n}\sum_{k=1}^{n}\E |X_{jk}^{(\alpha)}|^3
\end{equation}
The assumption (\ref{as1}) now yields
\begin{equation}
 \frac1{n^{\frac32}}\sum_{j=1}^{n}\sum_{k=1}^{n}\E |X_{jk}^{(\alpha)}|^3
|[\mathbf P^{(4)}]_{kj}|\le C\tau_n.
\end{equation}
Similar we get corresponding  bounds for  $\nu=5,6,7$
\begin{equation}
 \frac1{n^{\frac32}}\sum_{j=1}^{n}\sum_{k=1}^{n}\E |X_{jk}^{(\alpha)}|^3
|[\mathbf P^{(\nu)}]_{kj}|\le C\tau_n.
\end{equation}
and
\begin{equation}
 |L_{\mu}|\le C\tau_n,\quad \mu=1,\ldots,5.
\end{equation}

The bound of the quantity
\begin{equation}\label{fin1*}
 \widehat L_{\mu}=\sum_{j=1}^{n}\sum_{k=1}^{n}\E {X_{j,k}^{(\alpha)}}
\left[\frac{\partial \mathbf Q_{\nu}}{\partial X_{jk}^{(\alpha)}}\right]_{kj}.
\end{equation}
is similar.
Thus, the Lemma is proved.

\end{proof}
\begin{lem}\label{teilor}Under the  conditions of Theorem \ref{main} we have
\begin{equation}\notag
\sum_{j=1}^{n}\sum_{k=1}^{n}\E X_{jk}^{(\nu)}[\mathbf V_{\nu+1,m}\mathbf J\mathbf R
\mathbf V_{1,m-\nu+1}]_{kj}=\sum_{j=1}^{n}\sum_{k=1}^{n}
\E\left[\frac{\partial \mathbf V_{\nu+1,m}\mathbf J\mathbf R\mathbf V_{1,m-\nu+1}}
{\partial X_{jk}^{(\nu)}}\right]_{kj}+\varepsilon_n(z,\alpha))
\end{equation}
and
\begin{align}
\sum_{j=1}^{n}\sum_{k=1}^{n}&\E X^{(m-\nu+1)}_{j,k}[\mathbf V_{\nu+1,m}
\mathbf J\mathbf R\mathbf V_{1,m-\nu+1}]_{j+n,k}\notag\\&=\sum_{j=1}^{n}\sum_{k=1}^{n}
\E\left[\frac{\partial \mathbf V_{\nu+1,m}\mathbf J\mathbf R
\mathbf V_{1,m-\nu+1}}{\partial X_{jk}^{(\nu)}}\right]_{j+n,k}+\varepsilon_n(z,\alpha)),\notag
\end{align}
where $|\varepsilon_n(z,\alpha))|\le\frac{C\tau_n}{v^4}$.
\end{lem}
\begin{proof}
 By Taylor  expansion we have,
\begin{equation}\notag
 \E\xi f(\xi)=f'(0)\E\xi^2+\E(1-\theta)\xi^3f''(\theta\xi),
\end{equation}
and
\begin{equation}
  f'(0)=\E f'(\xi)-\E\xi f''(\theta\xi)
\end{equation}
where $\theta$ denotes a r.v.  which  uniformly  distributed on the unit
interval  and is independent on $\xi$.
After simple calculations we get
\begin{align}
 \sum_{j=1}^{n}\sum_{k=1}^{n}\E X_{jk}^{(\nu)}&[\mathbf V_{\nu+1,m}
\mathbf J\mathbf R\mathbf V_{1,m-\nu+1}]_{kj}=\sum_{j=1}^{n}\sum_{k=1}^{n}
\E\left[\frac{\partial \mathbf V_{\nu+1,m}\mathbf J\mathbf R\mathbf V_{1,m-\nu+1}}
{\partial X_{jk}^{(\nu)}}\right]_{kj}\notag\\
&+\sum_{j=1}^{n}\sum_{k=1}^{n}\E(-X_{jk}^{(\nu)}+(1-\theta_{jk}){X_{jk}^{(\nu)}}^3)
\left[\frac{\partial^2 \mathbf V_{\nu+1,m}\mathbf J\mathbf R\mathbf V_{1,m-\nu+1}}
{{\partial X_{jk}^{(\nu)}}^2}(\theta_{jk}^{(\nu)}
X_{jk}^{(\nu)})\right]_{kj}.\notag
\end{align}
Using the results of Lemma \ref{derivatives}, we conclude the proof.

\end{proof}

\end{document}